\documentclass
[
11pt]
{amsart}%

\usepackage{hyperref}
\usepackage{cite}
\usepackage{amssymb}
\usepackage{amsmath}
\usepackage{amsthm}
\usepackage{amsfonts}
\usepackage{graphicx}
\usepackage{bbm}
\usepackage{xcolor}
\usepackage[toc,page]{appendix}
\usepackage{subfigure}
\usepackage{bm} 
\usepackage[normalem]{ulem}
\usepackage{comment}
\usepackage{tikz, tikz-cd}
\usetikzlibrary{matrix,arrows,decorations.pathmorphing}

\newtheorem{theorem}{Theorem}[section]

\newtheorem{proposition}[theorem]{Proposition}

\newtheorem{lemma}[theorem]{Lemma}
\newtheorem{remark}[theorem]{Remark}
\newtheorem{example}[theorem]{Example}

\newtheorem*{acknowledgement}{Acknowledgement}
\theoremstyle{definition}
\newtheorem{definition}[theorem]{Definition}
\newtheorem{notation}[theorem]{Notation}

\numberwithin{equation}{section}

\begin{document}

\title[Log Sobolev on Heisenberg groups]{Logarithmic Sobolev inequalities on non-isotropic Heisenberg groups}

\author[Maria Gordina]{Maria Gordina{$^{\dag}$}}
\thanks{\footnotemark {$\dag$} Research was supported in part by NSF Grants DMS-1712427 and DMS-1954264.}
\address{$^{\dag}$ Department of Mathematics\\
University of Connecticut\\
Storrs, CT 06269,  U.S.A.}
\email{maria.gordina@uconn.edu}

\author[Liangbing Luo]{Liangbing Luo{$^{\dag}$}}
\address{$^{\dag}$ Department of Mathematics\\
University of Connecticut\\
Storrs, CT 06269,  U.S.A.}
\email{liangbing.luo@uconn.edu}




\begin{abstract}
We study logarithmic Sobolev inequalities with respect to a heat kernel measure on finite-dimensional and infinite-dimensional Heisenberg groups. Such a group is the simplest non-trivial example of a sub-Riemannian manifold.  First we consider logarithmic Sobolev inequalities on non-isotropic Heisenberg groups. These inequalities are considered with respect to the hypoelliptic heat kernel measure, and we show that the logarithmic Sobolev constants can be chosen to be independent of the dimension of the underlying space. In this setting, a natural Laplacian is not an elliptic but a hypoelliptic operator. The argument relies on comparing logarithmic Sobolev constants for the three-dimensional non-isotropic and isotropic Heisenberg groups, and tensorization of logarithmic Sobolev inequalities in the sub-Riemannian setting. Furthermore, we apply these results in an  infinite-dimensional setting and prove a logarithmic Sobolev inequality on an infinite-dimensional Heisenberg group modelled on an abstract Wiener space.
\end{abstract}

\maketitle

\tableofcontents

\section{Introduction}

The logarithmic Sobolev inequality  has been first introduced and studied by L. Gross in \cite{Gross1975c} on a Euclidean space with the Gaussian measure, and since then it found many applications. In particular, a number of existing results concern the question on how the constant in the logarithmic Sobolev inequality depends on the geometry of the underlying space, mostly in the Riemannian setting, see for example \cite[Section 5.7, Proposition 5.7.1]{BakryGentilLedouxBook}. The logarithmic Sobolev constant in that case  depends on the Ricci lower bound while it is independent of the dimension. The logarithmic Sobolev inequality is closely related to many important properties of the corresponding Markov semigroup such as hypercontractivity. Moreover, the fact that the logarithmic Sobolev constant often does not depend on the dimension makes it applicable in  infinite-dimensional settings.

Such results in the Riemannian setting rely on ellipticity of the Laplace-Beltrami operator as well as on geometric methods such as a curvature-dimension inequality, or different versions of  $\Gamma$ calculus. In the current paper we consider non-isotropic Heisenberg groups  which are the simplest non-trivial examples of sub-Riemannian manifolds.  The corresponding Laplacians are not elliptic operators but hypoelliptic which makes analysis more challenging. In addition, the Riemannian curvature-dimension condition  is not available. While recently such geometric methods have been developed for some sub-Riemannian manifolds starting with \cite{BaudoinGarofalo2017}, they are not easily applicable to non-isotropic Heisenberg groups of dimensions greater than $5$.

We consider a family of non-isotropic Heisenberg groups of a symplectic space $\left( \mathbb{R}^{2n}, \omega \right)$ defined as follows.

\begin{definition}\label{d.1.1}
A \emph{non-isotropic Heisenberg group} $\mathbb{H}^{n}_{\omega}$ is the set $\mathbb{R}^{2n}\times\mathbb{R}$ equipped with the group law given by

\begin{align}\label{GroupLaw}
& \left( \mathbf{v}, z \right)\star \left( \mathbf{v}^{\prime}, z^{\prime} \right)=\left(\mathbf{v}+\mathbf{v}^{\prime},z+z^{\prime}+\frac{1}{2}\omega\left( \mathbf{v}, \mathbf{v}^{\prime} \right)\right),
\\
& \mathbf{v}=\left( x_{1},y_{1},\cdots,x_{n},y_{n}  \right), \mathbf{v}^{\prime}=\left( x_{1}^{\prime},y_{1}^{\prime},\cdots,x_{n}^{\prime}, y_{n}^{\prime} \right) \in \mathbb{R}^{2n},
\notag
\\
& \omega: \mathbb{R}^{2n} \times \mathbb{R}^{2n} \longrightarrow \mathbb{R},
\notag
\end{align}
where
\begin{align}\label{e.SymplForm}
& \omega\left( \mathbf{v}, \mathbf{v}^{\prime} \right):=\sum_{i=1}^{n} \alpha_i\left(x_{i}y_{i}^{\prime}-x_{i}^{\prime}y_{i}\right)=\sum_{i=1}^{n} \omega_{i}\left(\mathbf{v}_{i}, \mathbf{v}_{i}^{\prime}\right),
\\
& \omega_{i}\left(\mathbf{v}_{i}, \mathbf{v}_{i}^{\prime}\right)=\alpha_i\left(x_{i}y_{i}^{\prime}-x_{i}^{\prime}y_{i}\right)
\notag
\\
& \mathbf{v}_{i}=\left( x_{i}, y_{i}\right), \mathbf{v}_{i}^{\prime}=\left( x_{i}^{\prime}, y_{i}^{\prime}\right)
\notag
\end{align}
is a  \emph{symplectic form} on $\mathbb{R}^{2n}$ and $\alpha_{1}, \alpha_{2}, \cdots, \alpha_{n}$ are positive constants indexed in such a way that
\begin{align*}
0<\alpha_{1}\leqslant \alpha_{2}\leqslant\cdots\leqslant \alpha_p=\alpha_{p+1}=\cdots=\alpha_{n}.
\end{align*}
\end{definition}
Note that any non-degenerate symplectic form on $\mathbb{R}^{2n}$, that is, a bilinear anti-symmetric form, can be written as a sum of symplectic forms on $\mathbb{R}^{2}$, as we describe in  Appendix~\ref{s.SymplSpace}. In particular, this explains why such groups are referred to as non-isotropic.

If $\alpha_{1}=\cdots=\alpha_{n}=1$, we get the standard $2n+1$-dimensional Heisenberg group. Sometimes the parametrization $\alpha_{1}=\cdots=\alpha_{n}=4$ is used for the standard Heisenberg group as in  \cite{BealsGaveauGreiner2000, LiHong-Quan2006, LiHong-QuanZhang2019} et al. These are all isotropic Heisenberg groups referring to the fact that the corresponding symplectic space is isotropic as described in Appendix~\ref{s.SymplSpace}.

We equip the group $\mathbb{H}^{n}_{\omega}$ with a sub-Riemannian manifold structure and the corresponding distance depending on the symplectic form $\omega$. The logarithmic Sobolev inequality we study is with respect to the heat kernel measure for the sub-Laplacian associated with the sub-Riemannian structure. One of the questions is how the logarithmic Sobolev constant depends on the symplectic form $\omega$ and the dimension of the Heisenberg group
$\mathbb{H}^{n}_{\omega}$.

Before describing our main result, let us review relevant mathematical literature. The logarithmic Sobolev inequality is known to hold in the isotropic case. For $n=1$ this inequality has been established by H.-Q.~Li  in  \cite{LiHong-Quan2006} with $\alpha_{1}=4$. His proof is based on pointwise upper and lower heat kernel estimates, and a gradient estimate known as the Driver-Melcher inequality. Motivated by \cite{Gross1992}  M. Bonnefont, D. Chafa\"i and R. Herry in \cite{BonnefontChafaiHerry2020} used a random walk approximation to study the case $n=1$. For  $n\geqslant 1$, W. Hebisch and B. Zegarlinski proved a logarithmic Sobolev inequality in \cite{HebischZegarlinski2010} using the tensorization property of logarithmic Sobolev inequalities and a lifting to the product group first introduced by \cite[Section 3]{Fraser2001a}. N.~Eldredge in \cite{Eldredge2010} proved the inequality on $H$-type groups using the hypoelliptic heat kernel estimates, such estimates on isotropic Heisenberg groups have been also shown in \cite{LiHong-Quan2007, HuLi2010}. Another approach to use H.-Q.~Li's heat kernel estimates to derive $L^{1}$ gradient bounds and a logarithmic Sobolev inequality has been used in \cite{BakryBaudoinBonnefontChafai2008}.

The measure considered in \cite{BakryBaudoinBonnefontChafai2008, BonnefontChafaiHerry2020, HebischZegarlinski2010, Eldredge2010} is the hypoelliptic heat kernel measure on $\mathbb{H}^n_{\omega}$ which can be regarded as an analogue of the Gaussian measure on the Euclidean space. In a different direction, \cite{BaudoinBonnefont2012} obtained a dimension-dependent upper bound for the logarithmic Sobolev constant with respect to the invariant measure of a subelliptic generator using a generalized curvature-dimension condition as developed in \cite{BaudoinGarofalo2017}.

F. Baudoin and Q. Feng in \cite{BaudoinFeng2015} used Malliavin's calculus to prove a version of logarithmic Sobolev inequalities on the horizontal path space with a constant depending on the geometry of the underlying space. In \cite{FrankLieb2012} R. Frank and L. Lieb proved a logarithmic Sobolev inequality on a Heisenberg group, with the measure being a Haar measure. They also show that  the logarithmic Sobolev constant  is sharp. In this case the logarithmic Sobolev constant is  dimension-dependent constant which is natural since they  use a Haar measure instead of the heat kernel measure that we consider in the current paper.

All of the results we mentioned previously apply only to the isotropic case. In the non-isotropic setting, one special case of non-isotropic Heisenberg groups was considered by E. Bou Dagher and B. Zegarlinski recently in a preprint  \cite{BouDagherZegarlinski2021}, in which they derived a dimension-dependent logarithmic inequality on such groups, but not for a heat kernel measure.

Moreover, the dependence of the logarithmic Sobolev constant on geometric characteristics of $\mathbb{H}^{n}_{\omega}$ has not been studied in  either isotropic or non-isotropic cases. Our main motivation for such a study is an application to infinite-dimensional Heisenberg-type groups introduced in \cite{DriverGordina2008} and studied in the sub-Riemannian setting in \cite{BaudoinGordinaMelcher2013, DriverEldredgeMelcher2016}, where non-isotropy is a consequence of the infinite-dimensional setting. This application is in spirit of the original use of a logarithmic Sobolev inequality but in a hypoelliptic infinite-dimensional setting.

Our paper is organized as follows. We first consider the case $n=1$ in Section~\ref{sec.LSINonisotropic.n=1}. Next, we study the tensorization argument of logarithmic Sobolev inequalities in the sub-Riemannian setting. Then we deduce the logarithmic Sobolev inequality on the non-isotropic Heisenberg group by regarding a non-isotropic Heisenberg group as a quotient group obtained from the product group. This allows us to use a dimension-independent constant in the logarithmic Sobolev inequality introduced  in Section~\ref{sec.LSINonisotropic}. Moreover, we show that the logarithmic Sobolev constant can be chosen to not depend on $\omega$ and the dimension. In Section~\ref{sec.OrderReversed}, we discuss a second approach when tensorization and lifting are reversed.

Finally, we apply the results on non-isotropic Heisenberg groups to the infinite-dimensional Heisenberg group with a  one-dimensional center in Section~\ref{sec.LSIInfinite}. While the classical finite-dimensional definition of hypoellipticity can not be directly used in this setting, it is known that the heat kernel measure is smooth by  \cite{BaudoinGordinaMelcher2013, DriverEldredgeMelcher2016}. Our results on the logarithmic Sobolev inequalities in the simplest infinite-dimensional hypoelliptic setting represent the next natural step in studying the logarithmic Sobolev inequalities  for infinite-dimensional hypoelliptic diffusions.

\section{Preliminaries} \label{sec.Preliminaries}

\subsection{Non-isotropic Heisenberg groups as sub-Riemannian manifolds}

A non-isotropic Heisenberg group $\mathbb{H}^{n}_{\omega}$ introduced in Definition~\ref{d.1.1} is a Lie group, with the identity being $e=\left( \mathbf{0}, 0 \right)$, and the inverse  given by $\left( \mathbf{v}, z \right)^{-1}= \left( -\mathbf{v}, -z \right)$. Its Lie algebra $\mathfrak{h}_{\omega}:=\mathcal{L}\left( \mathbb{H}^{n}_{\omega} \right) \cong T_e\mathbb{H}^{n}_{\omega}$ can be identified with the space  $\mathbb{R}^{2n+1}\cong \mathbb{R}^{2n} \times \mathbb{R}$  with the Lie bracket given by
\begin{equation}\label{e.LieBracket}
\left[
\left( \mathbf{a}_{1}, c_{1} \right), \left( \mathbf{a}_{2}, c_{2} \right)  \right] = \left(0,
\omega\left( \mathbf{a}_{1}, \mathbf{a}_{2} \right)  \right).
\end{equation}
The group $\mathbb{H}^{n}_{\omega}$ is a connected nilpotent group, and by \cite[Theorem 1.2.1]{CorwinGreenleafBook}  both the exponential and logarithmic maps are global diffeomorphisms. Thus the  exponential map $\exp: \mathfrak{h}_{\omega} \longrightarrow \mathbb{H}^{n}_{\omega}$, and its inverse map $\log: \mathbb{H}^n_{\omega} \longrightarrow \mathfrak{h}_{\omega}$ are well-defined on the whole Lie algebra  $\mathfrak{h}_{\omega}$ of $\mathbb{H}^{n}_{\omega}$. Moreover, we can describe them explicitly by
\[
\exp \left(\mathbf{a}, c\right)=\left(\mathbf{a}, c\right)
\]
for any $\left(\mathbf{a}, c\right)\in\mathfrak{h}_{\omega}$ and
\[
\log \left(\mathbf{v}, z\right)=\left(\mathbf{v}, z\right)
\]
for any $g=\left(\mathbf{v}, z\right)\in\mathbb{H}^{n}_{\omega}$. As a Carnot group $\mathbb{H}^{n}_{\omega}$ has a one-parameter group of automorphisms called   \emph{dilations}

\begin{align*}
& \delta_{\lambda}:\mathbb{H}^{n}_{\omega}\rightarrow \mathbb{H}^{n}_{\omega}, \lambda>0,
\\
& \delta_{\lambda}\left(\mathbf{v},z\right):=\left(\lambda\mathbf{v}, \lambda^2 z\right), \, g=\left(\mathbf{v},z\right)\in\mathbb{H}^{n}_{\omega}.
\end{align*}
We refer to \cite[Section 1.3]{BonfiglioliLanconelliUguzzoniBook} for more details.

Consider the following left-invariant vector fields on $\mathbb{H}^{n}_{\omega}$ identified with differential operators on $\mathbb{R}^{2n+1}$ by
\begin{align}\label{e.CanonicalBasis}
& X_i^{\omega}\left( g \right)=\frac{\partial}{\partial x_i}-\frac{\alpha_i}{2}y_i\frac{\partial}{\partial z}, \notag
\\
&
Y_i^{\omega}\left( g \right)=\frac{\partial}{\partial y_i}+\frac{\alpha_i}{2}x_i\frac{\partial}{\partial z}, \hskip0.1in i=1,\cdots,n
\\
&
Z^{\omega}\left( g \right)=\frac{\partial}{\partial z} \notag
\end{align}
for any $g=(x_{1},y_{1},\cdots,x_{n},y_{n},z)\in\mathbb{H}^{n}_{\omega}$. Note that the only non-zero Lie brackets for  left-invariant vector fields $X_i^{\omega}$ and $Y_i^{\omega}$ are
\[
[X_i^{\omega},Y_i^{\omega}]=\alpha_i Z^{\omega},  i=1,...,n,
\]
so the vector fields $\{X_i^{\omega}, Y_i^{\omega}, i=1, \cdots, n \}$ and their Lie brackets span the tangent space at every point, and therefore H\"{o}rmander's condition is satisfied.

This implies  that the group $\mathbb{H}^{n}_{\omega}$ has a natural sub-Riemannian structure $\left(\mathbb{H}^{n}_{\omega}, \mathcal{H}^{\omega}, \langle \cdot, \cdot \rangle^{\omega}_{\mathcal{H}}\right)$, where
\[
\mathcal{H}^{\omega}=\mathcal{H}_{g}^{\omega}= \operatorname{Span}\{X_i^{\omega} \left( g \right), Y_i^{\omega}\left( g \right), i=1,\cdots,n\}
\]
is the \emph{horizontal distribution} and the left-invariant inner product $\langle\cdot,\cdot\rangle_{\mathcal{H}^{\omega}}$ is chosen in such a way that  $\{X_i^{\omega},Y_i^{\omega}:i=1,\cdots,n\}$ is an orthonormal frame for the sub-bundle $\mathcal{H}^{\omega}$. Note that both the vector space $\mathcal{H}^{\omega}_g$ and the left-invariant sub-Riemannian metric $\langle \cdot, \cdot \rangle^{\omega}_{\mathcal{H}}=\langle \cdot, \cdot \rangle^{\omega}_{\mathcal{H}^{\omega}}$ depend on the symplectic form $\omega$.

We can equivalently describe the distribution $\mathcal{H}^{\omega}$ using a subspace of the Lie algebra $\mathfrak{h}_{\omega}$. Namely, if a \emph{horizontal space} $\mathcal{H}\subset \mathfrak{h}_{\omega}\cong T_{e}\mathbb{H}^{n}_{\omega}$ is equipped with the Euclidean inner product on $\mathbb{R}^{2n}$ with the corresponding norm denoted by $\vert \cdot \vert_{\mathcal{H}}$,  then we can use the left translation to define the sub-bundle $\mathcal{H}^{\omega}$ with the induced left-invariant sub-Riemannian metric $\langle \cdot, \cdot \rangle^{\omega}_{\mathcal{H}}$ and the corresponding left-invariant norm denoted by $\vert \cdot \vert_{\mathcal{H}^{\omega}}$ on $\mathcal{H}^{\omega}_g$ for any $g\in\mathbb{H}^{n}_{\omega}$. We will sometimes identify the horizontal distribution $\mathcal{H}^{\omega}$ and the horizontal space $\mathcal{H}$.

Recall that the Maurer-Cartan form $\theta$ on a Lie group $G$ is a Lie algebra-valued $1$-form defined by $\theta\left( v \right):=\theta_{g}\left( v \right)=L_{g^{-1} \ast}v$, $g\in G$, $v \in T_{g}G$.

\begin{definition} A path $\gamma: [a, b] \longrightarrow \mathbb{H}^{n}_{\omega}$ is said to be \emph{horizontal} if $\gamma$ is absolutely continuous and $\theta_{\gamma(t)}\left( \gamma^{\prime}(t) \right) \in \mathcal{H}$ for a.e. $t$. The \emph{length} of a horizontal path $\gamma: [a, b] \longrightarrow \mathbb{H}^{n}_{\omega}$ is defined to be
\[
l_{\mathcal{H}^{\omega}}\left( \gamma \right)=\int_{a}^{b} \vert \theta_{\gamma(t)}\left( \gamma^{\prime}(t) \right) \vert_{\mathcal{H}} dt.
\]
If $\gamma$ is not horizontal we define $l_{\mathcal{H}^{\omega}}\left( \gamma \right)=\infty$.

The \emph{Carnot-Carath\'eodory distance} between $g_{1}, g_{2} \in \mathbb{H}^{n}_{\omega}$ is defined as

\begin{align}\label{df.CCdistance}
d_{CC}^{\omega}(g_{1}, g_{2}):=\inf \left\{ l_{\mathcal{H}^{\omega}}\left( \gamma \right), \gamma\left( a \right)=g_{1}, \gamma\left( b \right)=g_{2} \right\}.
\end{align}
\end{definition}
The Chow-Rashevsky theorem (e.g. \cite[Section 19]{BonfiglioliLanconelliUguzzoniBook}) asserts  that H\"{o}rmander's condition implies that any two points in $\mathbb{H}^{n}_{\omega}$ can be joined by a horizontal path, therefore  $d_{CC}^{\omega}(g_{1}, g_{2})$ is finite for any $g_{1}, g_{2} \in \mathbb{H}^{n}_{\omega}$.

It is known that the infimum in \eqref{df.CCdistance} is attained, e.g. \cite[Theorem 5.15.5]{BonfiglioliLanconelliUguzzoniBook}. In addition, the Carnot-Carath\'eodory distance is a left-invariant metric on $\mathbb{H}^{n}_{\omega}$, that is, for any $g_{1},g_{2},g\in\mathbb{H}^{n}_{\omega}$
\begin{align*}
& d_{CC}^{\omega}(g_{1}, g_{2})=d_{CC}^{\omega}((g_{2})^{-1}g_{1}, e),
\\
& d_{CC}^{\omega}(g^{-1},e)=d_{CC}^{\omega}(g,e),
\end{align*}
e.g. \cite[Proposition 5.2.3, Proposition 5.2.4]{BonfiglioliLanconelliUguzzoniBook}.

\begin{notation} For any $g=\left(x_{1}, y_{1},\cdots,x_{n}, y_{n},z\right)\in\mathbb{H}^{n}_{\omega}$ we denote by
\[
d_{CC}^{\omega}(g):=d_{CC}^{\omega}(e,g)
\]
the corresponding norm.
\end{notation}
In addition to being left-invariant  $d_{CC}^{\omega}(g)$ is a homogeneous norm (e.g. \cite[Theorem 5.2.8]{BonfiglioliLanconelliUguzzoniBook}) and therefore
\[
d_{CC}^{\omega}\left(\delta_{\lambda} \left( g \right)\right)=\lambda d_{CC}^{\omega}\left( g \right), \lambda >0, g \in \mathbb{H}^{n}_{\omega}.
\]

\subsection{Sub-Laplacian and hypoelliptic heat kernel}

\begin{definition} \label{df.HorizontalGradient}
For any $f\in C^{\infty}(\mathbb{H}^{n}_{\omega})$, we let
\begin{align*}
\nabla_{\mathcal{H}}^{\omega}f=\nabla_{\mathcal{H}^{\omega}}^{\omega}f:= \sum_{i=1}^{n}
\left((X_i^{\omega}f)X_i^{\omega}+(Y_i^{\omega}f)Y_i^{\omega}\right)
\end{align*}
to be the \emph{horizontal gradient}.
\end{definition}

By the classical result in \cite{Hormander1967a} H\"{o}rmander's condition implies that the \emph{sub-Laplacian}
\begin{align} \label{df.SubLaplacian}
\Delta_{\mathcal{H}}^{\omega}=\Delta_{\mathcal{H}^{\omega}}^{\omega}:=\sum_{i=1}^{n} \left(\left(X_i^{\omega}\right)^2+\left(Y_i^{\omega}\right)^2\right)
\end{align}
is a hypoelliptic operator. For more on properties of $\Delta_{\mathcal{H}}^{\omega}$ in a more general setting we refer to \cite[Section 3]{DriverGrossSaloff-Coste2009a}, some of which we describe below. In particular, the sub-Laplacian only depends on the sub-Riemannian metric $\langle \cdot,\cdot\rangle^{\omega}_{\mathcal{H}^{\omega}}$ but it is independent of the choice of orthonormal frame by \cite[Theorem 3.8]{GordinaLaetsch2016a}.

Next, we define the hypoelliptic heat kernel measure on $\mathbb{H}^n_{\omega}$. First we choose a bi-invariant Haar measure $dg$ on $\mathbb{H}^{n}_{\omega}$ to be the Lebesgue measure
\[
dg=dx_{1}dy_{1}\cdots dx_{n}dy_{n}dz
\]
on $\mathbb{R}^{2n+1}$. The sub-Laplacian $\Delta_{\mathcal{H}}^{\omega}$ is essentially self-adjoint on $C_{c}^{\infty}\left( \mathbb{H}^{n}_{\omega} \right)$ in $L^{2}\left( \mathbb{H}^{n}_{\omega}, dg \right)$. The corresponding semigroup by $e^{t \Delta_{\mathcal{H}}^{\omega}/2}$  admits a probability transition kernel $\mu_t^{\omega}\left( g, dh \right)$ such that $\mu_t^{\omega}\left( g, A \right)\geqslant 0$ for all Borel sets $A$ and
\[
\left( e^{t \Delta_{\mathcal{H}}^{\omega}/2}f \right)\left( g \right)=\int_{\mathbb{H}^{n}_{\omega}} f\left( h \right) \mu_t^{\omega}\left( g, dh \right)
\]
for all $f \in L^{2}\left( \mathbb{H}^{n}_{\omega}, dg \right)$.

As explained at \cite[p. 952]{DriverGrossSaloff-Coste2009a} the transition kernel measure $\mu_t^{\omega}\left( g, dh \right)$ admits a continuous density, $p_t^{\omega}\left( g, h \right)$, with respect to the Haar measure $dg$

\begin{align}\label{eqn.HeatKernelMeasureDF}
\mu_t^{\omega}\left( g, dh \right)=p_{t}^{\omega}\left( g, h \right)dh.
\end{align}
Note that the sub-Laplacian $\Delta_{\mathcal{H}}^{\omega}$ commutes with left translations which together with bi-invariance of the Haar measure imply that

\begin{align} \label{eqn.LeftInvariance}
p_{t}^{\omega}\left( g, h \right)=p_{t}^{\omega}\left( e, g^{-1}h \right),
\end{align}
therefore it suffices to look at the function $p_{t}^{\omega}\left( e, g\right)$. From now on we use $p_{t}^{\omega}\left( g \right)$ to denote this function and we will refer to it as the \emph{heat kernel}.

\begin{remark}
An explicit formula for $p_t^{\omega}(g)$ is
\begin{align} \label{eqn.HeatKernelFormulaNonisotropic}
& p_t^{\omega}(g)=p_t^{\omega}\left(\mathbf{v}_1,\cdots,\mathbf{v}_n,z\right)
\\
& = \frac{1}{(2\pi t)^{n+1}}\int_{\mathbb{R}} e^{\frac{1}{t}\left(2izs -\sum_{j=1}^n\frac{\alpha_js}{2} \coth\left(\alpha_js\right) \Vert \mathbf{v}_j\Vert^2\right)} \prod_{j=1}^n\left( \frac{\alpha_js}{\sinh\left(\alpha_js \right)}\right) ds
\notag
\end{align}
for any $g=\left(\mathbf{v}_1,\cdots,\mathbf{v}_n,z\right)=(x_{1},y_{1},\cdots,x_{n},y_{n},z)\in\mathbb{H}^{n}_{\omega}$ with $\mathbf{v}_j=(x_j,y_j)$ for $j=1,\cdots,n$ and $\Vert \cdot \Vert$ is the Euclidean norm on $\mathbb{R}^2$; see for example \cite{LiHong-QuanZhang2019}. By \eqref{eqn.HeatKernelFormulaNonisotropic} we see that
\begin{align} \label{eqn.TimeHomogeneity}
p_t^{\omega}\left(\delta_{\lambda}(g)\right)=\frac{1}{\lambda^{2(n+1)}} p_{\frac{t}{\lambda^2}}^{\omega}(g), \, g\in \mathbb{H}^{n}_{\omega}.
\end{align}
\end{remark}

\begin{definition} \label{df.HeatKernelMeasureNonisotropic}
We call a family of measures $\{\mu_t^{\omega}\}_{t>0}$ on $\mathbb{H}^{n}_{\omega}$ with
\[
d\mu_t^{\omega}\left(  g \right)=\mu_t^{\omega}\left(  dg \right)=\mu_t^{\omega}\left( e, dg \right)=p_{t}^{\omega}\left( g \right)dg
\]
the \emph{heat kernel measure}.
\end{definition}

By \cite[Theorem 3.4 (ii)]{DriverGrossSaloff-Coste2009a}, $\{\mu_t^{\omega}\}_{t>0}$ is a family of probability measures. In addition \cite[Theorem 6.15]{DriverGrossSaloff-Coste2010} gives an equivalent way of defining the heat kernel measure, which we will use later. For completeness, we include its statement below.

\begin{proposition}[Theorem 6.15 in \cite{DriverGrossSaloff-Coste2010}] \label{prop.HeatKernelMeasureDF2}
$\{\mu_t^{\omega}\}_{t>0}$ is the unique family of probability measures on $\mathbb{H}^n_{\omega}$ that satisfies the heat equation as follows
\begin{align} \label{eqn.HeatEquation}
& \frac{d}{dt} \int_{\mathbb{H}^{n}_{\omega}}f\left( g \right)d\mu_t^{\omega}\left(g \right)=\int_{\mathbb{H}^{n}_{\omega}}\left(\frac{1}{2}\Delta_{\mathcal{H}}^{\omega}f\right)\left( g \right)d\mu_t^{\omega}\left( g \right),
\\
&
\lim_{t \to 0}\int_{\mathbb{H}^{n}_{\omega}}f\left( g \right)d\mu_t^{\omega}\left( g \right)=f(e)
\notag
\end{align}
for any $t>0$ and any $f\in C^{\infty}_c\left(\mathbb{H}^{n}_{\omega}\right)$.
\end{proposition}

\begin{definition}
We say that $\mathbb{H}^{n}_{\omega}$ with the heat kernel measure $\mu_t^{\omega}$ satisfies a \emph{logarithmic Sobolev inequality with constant $C\left(\omega, t\right)$} if
\begin{align} \label{LSI}
& \int_{\mathbb{H}^{n}_{\omega}}f^2\log f^2d\mu_t^{\omega}-\left(\int_{\mathbb{H}^{n}_{\omega}}f^2 d\mu_t^{\omega}\right)\log\left(\int_{\mathbb{H}^{n}_{\omega}}f^2d\mu_t^{\omega}\right)
\\
&
\leqslant C\left(\omega, t\right)\int_{\mathbb{H}^{n}_{\omega}} \vert \nabla_{\mathcal{H}}^{\omega}f\vert_{\mathcal{H}^{\omega}}^2 d\mu_t^{\omega}
\notag
\end{align}
for any $f\in C^{\infty}_c\left(\mathbb{H}^{n}_{\omega}\right)$ and any $t>0$. In such a case we also say that $\operatorname{LSI}\left(C\left(\omega, t \right), \mu_t^{\omega}\right)$ holds.
\end{definition}

As we mentioned in the introduction, the logarithmic Sobolev inequality with respect to the heat kernel measure is known to hold in the isotropic case, both for $n=1$ and $n>1$, and we include the result for $n=1$ for a later reference. For $n>1$ we refer to \cite[Theorem 7.3]{HebischZegarlinski2010}. In the statement below we denote the standard symplectic form on $\mathbb{R}^2$ by $\omega_0$, and the corresponding $3$-dimensional isotropic Heisenberg group by $\mathbb{H}^1_{\omega_0}$.

\begin{theorem}[Corollaire 1.2 in \cite{LiHong-Quan2006}]\label{t.HQLi} There is a constant $C\left(\omega_{0}, t\right) \in \left( 0, \infty \right)$ such that
\begin{align*}
& \int_{\mathbb{H}^{1}_{\omega_{0}}}f^2\log f^2d\mu_t^{\omega_{0}}-\left(\int_{\mathbb{H}^{1}_{\omega_{0}}}f^2 d\mu_t^{\omega_{0}}\right)\log\left(\int_{\mathbb{H}^{1}_{\omega_{0}}}f^2d\mu_t^{\omega_{0}}\right)
\\
&
\leqslant C\left(\omega_{0}, t\right)\int_{\mathbb{H}^{1}_{\omega_{0}}} \vert \nabla_{\mathcal{H}}^{\omega_{0}}f\vert_{\mathcal{H}^{\omega_{0}}}^2 d\mu_t^{\omega_{0}}
\notag
\end{align*}
for any $f\in C^{\infty}_c\left(\mathbb{H}^{1}_{\omega_{0}}\right)$ and any $t>0$.
\end{theorem}

\begin{remark}\label{rem.HQLi} In addition to the statement above H.-Q.~Li proved that

\[
C\left(\omega_{0}, t\right)=C^{2}t,
\]
where $C$ is the constant in the Driver-Melcher inequality \cite[Th\'{e}or\`{e}me 1.1]{LiHong-Quan2006} proved originally in \cite{DriverMelcher2005} for $p>1$. To the best of our knowledge there is no sharpness result for this inequality.
\end{remark}

First we reduce consideration of the logarithmic Sobolev inequalities on $\mathbb{H}^{n}_{\omega}$  by relying on time-homogeneity of the heat kernel \eqref{eqn.TimeHomogeneity} to concentrate on the case of $t=1$.

\begin{proposition} \label{prop.LSITimeScaling}
Suppose  $\mathbb{H}^{n}_{\omega}$ is an non-isotropic Heisenberg group, then if $\operatorname{LSI}\left(C\left(\omega\right),\mu^{\omega}_{1}\right)$ holds, then  $\operatorname{LSI}\left(C\left(\omega, t \right),\mu_t^{\omega}\right)$ holds for any $t>0$, where  $C\left( \omega, t \right)=C\left(\omega\right)t$.

\end{proposition}

\begin{proof}
Suppose $f\in C^{\infty}_{c}\left(\mathbb{H}^{n}_{\omega}\right)$ and $t>0$, then we have $f\circ \delta_{\sqrt{t}}\in C^{\infty}_c\left(\mathbb{H}^{n}_{\omega}\right)$, and therefore
\begin{align} \label{e.2.11}
& \int_{\mathbb{H}^{n}_{\omega}}\left(f\circ \delta_{\sqrt{t}}\right)^2\log\left(f\circ \delta_{\sqrt{t}}\right)^2d\mu_{1}^{\omega}
\notag
\\
& -\left(\int_{\mathbb{H}^{n}_{\omega}}(f\circ \delta_{\sqrt{t}})^2d\mu_{1}^{\omega}\right)\log\left(\int_{\mathbb{H}^{n}_{\omega}}(f\circ \delta_{\sqrt{t}})^2d\mu_{1}^{\omega}\right)
\\
&
\leqslant C\left(\omega\right)\int_{\mathbb{H}^{n}_{\omega}}\left\vert \nabla_{\mathcal{H}}^{\omega}(f\circ \delta_{\sqrt{t}})\right\vert_{\mathcal{H}^{\omega}}^2d\mu_{1}^{\omega}.
\notag
\end{align}
Now we can use \eqref{eqn.TimeHomogeneity} with $\lambda=\sqrt{t}$ to see that  $d\mu_{1}^{\omega}=t^{n+1}p_t^{\omega}\left(\delta_{\sqrt{t}}(g)\right)dg$. Then  \eqref{LSI} follows with $C\left(\omega,t\right)=C\left(\omega\right)t$ by using the change of variables $\delta_{\sqrt{t}}(g) \mapsto g$ in \eqref{e.2.11}.
\end{proof}

\section{Logarithmic Sobolev inequalities on $\mathbb{H}^1_{\omega}$} \label{sec.LSINonisotropic.n=1}

\subsection{Comparison between isotropic and non-isotropic Heisenberg groups for $n=1$}

For   $\mathbb{H}^1_{\omega}\cong \mathbb{R}^{3}$ the group law defined by \eqref{GroupLaw} can be written as follows
\begin{align}\label{GroupLaw.n=1}
& \left( x_{1}, y_{1}, z_{1} \right)\star \left( x_{2}, y_{2}, z_{2} \right)
\notag
\\
& =\left( x_{1}+x_{2}, y_{1}+y_{2}, z_{1}+z_{2} + \frac{\alpha}{2}\left( x_{1}y_{2}- x_{2}y_{1}\right)\right)
\end{align}
for some $\alpha>0$. The isotropic Heisenberg group, $\mathbb{H}^1_{\omega_0}$, corresponds to   $\alpha=1$. The difference between $\mathbb{H}^1_{\omega}$ and $\mathbb{H}^1_{\omega_0}$ is the symplectic form $\omega$ on $\mathbb{R}^2$ which in this case is parameterized by $\alpha$ in \eqref{e.SymplForm}. Our goal in this section is to study how the logarithmic Sobolev constant on $\mathbb{H}^1_{\omega}$ depends on the parameter $\alpha$ by comparing the non-isotropic and isotropic cases.

Consider the map

\begin{align}\label{e.SubRiemIsometry}
& F: \mathbb{H}^1_{\omega_0} \rightarrow \mathbb{H}^1_{\omega},
\notag
\\
& F(g)=F\left( x, y, z \right):=\left( x, y, \alpha z \right), g=\left( x, y, z \right) \in \mathbb{H}^1_{\omega_0}.
\end{align}
The next statement shows that we can view $F$ as an isomorphism between sub-Riemannian manifolds.

\begin{lemma}[Comparison between $\mathbb{H}^1_{\omega_0}$ and $\mathbb{H}^1_{\omega}$] \label{lem.ComparisonLemma}
The map $F$ is a Lie group isomorphism commuting with the left translation $L_g$, namely,
\begin{align} \label{T}
F\circ L_{g}=L_{F(g)}\circ F \text{ for any } g\in\mathbb{H}^1_{\omega_0}.
\end{align}
The restriction of the differential of $F$ to each fiber of the horizontal distribution at any $g \in \mathbb{H}^1_{\omega_0}$, $dF_g\vert_{\mathcal{H}^{\omega_0}_{g}}: \mathcal{H}^{\omega_0}_{g} \rightarrow \mathcal{H}^{\omega}_{F(g)}$, is an isometry. Moreover, for any $f\in C^{\infty}_c\left(\mathbb{H}^{1}_{\omega}\right)$
\begin{align} \label{eqn.GradientLengthRelation(1)}
\vert \nabla^{\omega_0}_{\mathcal{H}}(f\circ F)\vert_{\mathcal{H}^{\omega_0}}=\vert \nabla^{\omega}_{\mathcal{H}}f\vert_{\mathcal{H}^{\omega}} \circ F.
\end{align}
The pushforward of the heat kernel measure $\mu_t^{\omega_0}$ by $F$ is the heat kernel measure $\mu_t^{\omega}$ on $\mathbb{H}^1_{\omega}$.
\end{lemma}

\begin{proof}
Equation~\ref{T} follows directly from  the multiplication law \eqref{GroupLaw.n=1}.

Using explicit formulas for the exponential map and $F$, we see that the differential of $F$ at $e$, $dF_{e}: T_e\mathbb{H}^1_{\omega_0} \rightarrow T_{F(e)}\mathbb{H}^1_{\omega}$ is given by
\begin{align}\label{e.3.5}
dF_e(\mathbf{a}, c)=(\mathbf{a},\alpha c), \,   (\mathbf{a}, c)\in   T_e\mathbb{H}^1_{\omega_0},
\end{align}
which shows that $dF_{e}$ is a bijective linear transformation. By \eqref{e.LieBracket} we see that the Lie brackets are preserved under $dF_{e}$, and thus $dF_{e}$ is a Lie algebra isomorphism. For connected and simply connected Lie groups $\mathbb{H}^1_{\omega_0}$ and $\mathbb{H}^1_{\omega}$, the map $F$ is a Lie group isomorphism by \cite[Corollary 3.8]{HallLieBook}.

At the identity we have $\mathcal{H}^{\omega_0}_e=\mathcal{H}^{\omega}_e=\mathbb{R}^2 \subset \mathbb{R}^3$ and the differential of $F$ at $e$ restricted to $\mathcal{H}^{\omega_0}_e$ is the identity map on $\mathbb{R}^2$ by \eqref{e.3.5}, thus we have $dF_{e}\vert_{\mathcal{H}^{\omega_0}_e}: \mathcal{H}^{\omega_0}_e \rightarrow \mathcal{H}^{\omega}_e$. Note that both $\langle \cdot,\cdot\rangle^{\omega_0}_{\mathcal{H}_e}$ and $\langle \cdot,\cdot\rangle^{\omega}_{\mathcal{H}_e}$ are the same Euclidean inner products, so $dF_{e}\vert_{\mathcal{H}^{\omega_0}_e}:\mathcal{H}^{\omega_0}_e \rightarrow \mathcal{H}^{\omega}_e$ is an isometry.

Now we can use the fact that $F$ commutes with the left multiplication by \eqref{T} to extend this to fibers of the horizontal distribution at any $g \in \mathbb{H}^1_{\omega_0}$. Consider $dF_g\vert_{\mathcal{H}^{\omega_0}_g}: \mathcal{H}^{\omega_0}_g \rightarrow \mathcal{H}^{\omega}_{F(g)}$, then  left-invariance of sub-Riemannian metrics on these groups and \eqref{T} imply that $dF_g\vert_{\mathcal{H}^{\omega_0}_g}: \mathcal{H}^{\omega_0}_g \rightarrow \mathcal{H}^{\omega}_{F(g)}$ is an isometry on sub-Riemannian distributions.

Next, we see that the orthonormal frames $\{X^{\omega}, Y^{\omega}\}$ on  $\left(\mathbb{H}^1_{\omega},\mathcal{H},\langle \cdot,\cdot\rangle^{\omega}_{\mathcal{H}}\right)$ and $\{X^{\omega_0},Y^{\omega_0}\}$ on  $\left(\mathbb{H}^1_{\omega_0},\mathcal{H},\langle \cdot,\cdot\rangle^{\omega_0}_{\mathcal{H}}\right)$ introduced by \eqref{e.CanonicalBasis} satisfy
\begin{align}
& (dF_g)\left(X^{\omega_0}(g)\right)=X^{\omega}(F(g)), \label{eqn.VectorRelation(1)1}
\\
&
(dF_g)\left(Y^{\omega_0}(g)\right)=Y^{\omega}(F(g)). \label{eqn.VectorRelation(1)2}
\end{align}
For any $f\in C^{\infty}_c\left(\mathbb{H}^1_{\omega}\right)$, by \eqref{eqn.VectorRelation(1)1} and \eqref{eqn.VectorRelation(1)2} we have
\begin{align*}
\left(\nabla_{\mathcal{H}}^{\omega}f\right)(F(g)) & =\left((X^{\omega}f)(F(g))\right)X^{\omega}(F(g))+\left((Y^{\omega}f)(F(g))\right)Y^{\omega}(F(g))
\\
&
=(dF_g)\left(\left(\nabla^{\omega_0}_{\mathcal{H}}(f\circ F)\right)(g)\right),
\end{align*}
therefore Equation~\ref{eqn.GradientLengthRelation(1)} follows since $dF_g\vert_{\mathcal{H}^{\omega_0}_g}: \mathcal{H}^{\omega_0}_g \rightarrow \mathcal{H}^{\omega}_{F(g)}$ is an isometry.

Finally, we compute the pushforward measure $F_{\#}\mu_t^{\omega_0}$. For any Borel set $E$ on $\mathbb{H}^{1}_{\omega}$, change of variable gives
\begin{align*}
\int_{F^{-1}(E)}d\mu_t^{\omega_0}=\int_{E} \frac{p_t^{\omega_0}\left(F^{-1}(g)\right)}{\alpha}dg,
\end{align*}
so $F_{\#}\mu_t^{\omega_0}$ has the form
\[
d\left(F_{\#}\mu_t^{\omega_0}\right)=\frac{p_t^{\omega_0}\left(F^{-1}(g)\right)}{\alpha}dg.
\]
Alternatively one can check that $F_{\#}\mu_t^{\omega_0}$ satisfies \ref{eqn.HeatEquation} using \ref{eqn.VectorRelation(1)1} and \ref{eqn.VectorRelation(1)2} without knowing its explicit formula as above. By the explicit formula for $p_t^{\omega}$ on $\mathbb{H}^1_{\omega}$ and Definition~\ref{df.HeatKernelMeasureNonisotropic}, we can show that the pushforward measure $F_{\#}\mu_t^{\omega_0}$ is the heat kernel measure $\mu_t^{\omega}$ on $\mathbb{H}^{1}_{\omega}$.
\end{proof}

\subsection{Logarithmic Sobolev inequalities on $\mathbb{H}^1_{\omega}$}

We start by recalling that by Theorem~\ref{t.HQLi} the logarithmic Sobolev inequality  $\operatorname{LSI}\left(C\left( \omega_0\right), \mu_1^{\omega_0} \right)$ holds on the isotropic Heisenberg group $\mathbb{H}^1_{\omega_{0}}$.

\begin{theorem}\label{thm.LSINonisotropic.n=1}
The logarithmic Sobolev inequality $\operatorname{LSI}(C\left(\omega\right)t,\mu_t^{\omega})$ holds on $\mathbb{H}^1_{\omega}$ with the logarithmic Sobolev constant $C\left(\omega\right)=C\left(\omega_{0}\right)$, the constant for the isotropic Heisenberg group $\mathbb{H}^1_{\omega_{0}}$,  and thus $C\left(\omega\right)$ can be chosen to be independent of $\omega$.
\end{theorem}

\begin{proof}
By \cite[Theorem 7.3]{HebischZegarlinski2010} the isotropic Heisenberg group $\mathbb{H}^1_{\omega_{0}}$ satisfies a logarithmic Sobolev inequality   $\operatorname{LSI}\left(C\left(\omega_0\right),\mu_1^{\omega_0}\right)$. Note that by Proposition~\ref{prop.LSITimeScaling}, it suffices to compare the constants at time $t=1$, that is, to find constants in  $\operatorname{LSI}\left(C\left(\omega\right), \mu_1^{\omega}\right)$ and $\operatorname{LSI}\left(C\left(\omega_0\right),\mu_1^{\omega_0}\right)$.

For any $f \in C^{\infty}_c\left(\mathbb{H}^1_{\omega}\right)$, we have $f\circ F\in C^{\infty}_c\left(\mathbb{H}^1_{\omega_0}\right)$, where $F$ is the map defined by \eqref{e.SubRiemIsometry}. Then we can use the logarithmic Sobolev inequality $\operatorname{LSI}\left(C\left(\omega_0\right),\mu_1^{\omega_0}\right)$  on $\mathbb{H}^1_{\omega_0}$ for $f\circ F$ to see that
\begin{align*}
& \int_{\mathbb{H}^1_{\omega_0}}(f\circ F)^2\log (f\circ F)^2d\mu_1^{\omega_0}-\left(\int_{\mathbb{H}^1_{\omega_0}}(f\circ F)^2d\mu_1^{\omega_0}\right)\log\left(\int_{\mathbb{H}^1_{\omega_0}}(f\circ F)^2d\mu_1^{\omega_0}\right) \notag
\\
&
\leqslant C\left(\omega_0\right)\int_{\mathbb{H}^1_{\omega_0}} \vert \nabla_{\mathcal{H}}^{\omega_0}(f\circ F)\vert_{\mathcal{H}^{\omega_0}}^2 d\mu_1^{\omega_0}.
\end{align*}
Using the change of variables $F(g)\mapsto g$ in this inequality together with Lemma~\ref{lem.ComparisonLemma}, we see that $\operatorname{LSI}\left(C\left(\omega_0\right),\mu_1^{\omega}\right)$ holds on $\mathbb{H}^1_{\omega}$, which implies that we can take $C\left(\omega\right)=C\left(\omega_0\right)$.
\end{proof}

\begin{remark}
Note that the argument used in the proof of Theorem~\ref{thm.LSINonisotropic.n=1} shows that if there is an optimizer for $\operatorname{LSI}\left(C\left(\omega_0\right),\mu_1^{\omega_{0}}\right)$, we can find an optimizer for $\operatorname{LSI}\left(C\left(\omega\right),\mu_1^{\omega}\right)$ using the change of variables. In this case, $\operatorname{LSI}\left(C\left(\omega\right)t,\mu_t^{\omega}\right)$ and $\operatorname{LSI}\left(C\left(\omega_0\right)t,\mu_t^{\omega_0}\right)$ would have the same optimal constant $C\left(\omega_0\right)$, which is independent of the symplectic form $\omega$ as well.
\end{remark}

\begin{remark}
The map $F$ can be regarded as a scaling of the metric on the horizontal space $\mathcal{H}$. Indeed, a scaling of an orthogonal symplectic basis as described by Proposition~\ref{p.SymplBasis} is equivalent to changing parameters $\alpha_{1}, ..., \alpha_{n}$.
Thus Theorem~\ref{thm.LSINonisotropic.n=1} shows that the logarithmic Sobolev constant (and the optimal one if it exists) in the case $n=1$ is independent of the sub-Riemannian metric we equip $\mathbb{R}^3$. Note that this metric can be thought of as a scaling of a symplectic basis in Theorem \ref{t.SymplBasis}. The fact that the logarithmic Sobolev constant is independent of $\omega$ is not surprising when compared to the phenomenon for the Gaussian measure on a Euclidean space equipped with a Riemannian metric corresponding to the covariance of the Gaussian measure.
\end{remark}

\section{Logarithmic Sobolev Inequalities on $\mathbb{H}^{n}_{\omega}$} \label{sec.LSINonisotropic}

\subsection{Tensorization in the sub-Riemannian setting}

Tensorization is a fundamental property of logarithmic Sobolev inequalities, that is, the logarithmic Sobolev inequality  holds on the product space of two probability spaces each of which satisfy a logarithmic Sobolev inequality (e.g. \cite[Theorem 4.4]{GuionnetZegarlinski2003}). Here we include a version of the tensorization of logarithmic Sobolev inequalities on the product group of three-dimensional non-isotropic Heisenberg groups.

Let $\{\mathbb{H}^1_{\omega_j}\}_{j=1}^{n}$ be a family of $3$-dimensional Heisenberg groups. Then the product group $\mathbb{H}^1_{\omega_1}\times\cdots\times \mathbb{H}^1_{\omega_n}$ is a sub-Riemannian manifold with the horizontal sub-bundle $\mathcal{H}:=\mathcal{H}^{\omega_1} \oplus \cdots \oplus \mathcal{H}^{\omega_n}$ defined fiberwise and the corresponding metric $\langle\cdot,\cdot\rangle_{\mathcal{H}}$ and norm $\vert \cdot\vert_{\mathcal{H}}$. The sub-Laplacian for the product group $\mathbb{H}^1_{\omega_{1}} \times \cdots \times \mathbb{H}^1_{\omega_{n}}$ is  $\sum_{j=1}^{n}\Delta^{\omega_j}_{\mathcal{H}}$ which is an operator on the product space as considered in  \cite[Proposition 18]{Schechtman2003a}. For any $f\in C^{\infty}_c\left(\mathbb{H}^1_{\omega_{1}} \times \cdots \times \mathbb{H}^1_{\omega_{n}}\right)$, we denote the horizontal gradient by $\nabla_{\mathcal{H}}f$. Finally the corresponding heat kernel measure $\mu_t$ is the product measure $\mu_t^{\omega_{1}} \otimes\cdots\otimes \mu_t^{\omega_{n}}$.

Applying \cite[Proposition 18]{Schechtman2003a} together with Theorem~\ref{thm.LSINonisotropic.n=1} gives the following results for the product group $\mathbb{H}^1_{\omega_{1}} \times \cdots \times \mathbb{H}^1_{\omega_{n}}$.

\begin{proposition} \label{prop.LSITensorization}
The product group $\mathbb{H}^1_{\omega_1}\times \cdots\times \mathbb{H}^1_{\omega_n}$ satisfies a logarithmic Sobolev inequality $\operatorname{LSI}\left( C\left(\omega_1,\cdots,\omega_n,t\right), \mu_t^{\omega_{1}} \otimes\cdots\otimes \mu_t^{\omega_{n}}\right)$, where the constant can be chosen to be $C\left(\omega_1,\cdots,\omega_n,t\right)=C\left(\omega_0\right)t$ which is independent of the symplectic forms $\omega_1,\cdots, \omega_n$ and $n$.
\end{proposition}

\subsection{From the product group to a non-isotropic Heisenberg group} \label{sec.LSI}

Given $\mathbb{H}^{n}_{\omega}$, we consider $\mathbb{H}^1_{\omega_{1}} \times \cdots \times \mathbb{H}^1_{\omega_{n}}$, where $\omega_i$ for $i=1,\cdots, n$ are symplectic forms on $\mathbb{R}^{2}$ defined by \eqref{e.SymplForm}. The following construction was introduced in \cite[Section 3]{Fraser2001a} and used by \cite[Theorem 7.3]{HebischZegarlinski2010} for the isotropic Heisenberg group $\mathbb{H}^n_{\omega_0}$. Here we extend it to the non-isotropic case. Define
\begin{align}\label{pi}
& \pi:\mathbb{H}^1_{\omega_{1}} \times \cdots \times \mathbb{H}^1_{\omega_{n}} \rightarrow \mathbb{H}^{n}_{\omega}, \notag
\\
& \pi(g_1, \cdots, g_n):=\pi(x_{1}, y_{1}, z_{1}, \cdots, x_{n}, y_{n}, z_{n})=(x_{1}, y_{1}, \cdots, x_{n}, y_{n}, z),
\\
&
z=\sum_{i=1}^{n} z_i \notag
\end{align}
for any $(g_1,\cdots,g_n)\in \mathbb{H}^1_{\omega_{1}} \times \cdots \times \mathbb{H}^1_{\omega_{n}}$. The next statement shows that we can view $\pi$ as a homomorphism between sub-Riemannian manifolds.

\begin{proposition}[$\mathbb{H}^{n}_{\omega}$ and $\mathbb{H}^1_{\omega_{1}} \times \cdots \times \mathbb{H}^1_{\omega_{n}}$] \label{p.QuotientLemma}
The map $\pi$  is a Lie group homomorphism commuting with the left translation $L_{(g_1,\cdots,g_n)}$ on the product group as follows
\begin{align} \label{pi1}
& \pi\circ L_{(g_1,\cdots,g_n)}=L_{\pi(g_1,\cdots,g_n)} \circ \pi,
\\
& (g_1,\cdots,g_n)\in \mathbb{H}^1_{\omega_{1}} \times \cdots \times \mathbb{H}^1_{\omega_{n}}.
\notag
\end{align}
The restriction of the differential of $\pi$  to horizontal spaces,
\[
d\pi_{(g_1,\cdots,g_n)}\vert_{\mathcal{H}_{(g_1,\cdots,g_n)}}: \mathcal{H}_{(g_1,\cdots,g_n)} \rightarrow \mathcal{H}^{\omega}_{\pi(g_1, \cdots, g_n)}
\]
is an isometry. Moreover, for any $f\in C^{\infty}_c\left(\mathbb{H}^{n}_{\omega}\right)$
\begin{align} \label{eqn.GradientLengthRelation(2)}
\vert \nabla_{\mathcal{H}}(f\circ \pi)\vert_{\mathcal{H}}=\vert \nabla^{\omega}_{\mathcal{H}}f\vert_{\mathcal{H}^{\omega}} \circ \pi.
\end{align}
In addition, the pushforward by $\pi$ of the heat kernel measure $\mu_t$ on $\mathbb{H}^1_{\omega_{1}} \times \cdots \times \mathbb{H}^1_{\omega_{n}}$ is the heat kernel measure $\mu_t^{\omega}$ on $\mathbb{H}^n_{\omega}$.
\end{proposition}

\begin{proof}
Equation \eqref{pi1} follows directly from the multiplication law given by \eqref{GroupLaw}.

Recall again that for connected and simply connected Lie groups to show that a map is a Lie group homomorphism it is enough to check that its differential at the identity is a Lie algebra homomorphism by \cite[Theorem 3.7]{HallLieBook}. Applying this to $\mathbb{H}^1_{\omega_{1}} \times \cdots \times \mathbb{H}^1_{\omega_{n}}$ and $\mathbb{H}^n_{\omega}$, we see it is enough to check that the differential of $\pi$ at the identity is a Lie algebra homomorphism between $T_{(e_1,\cdots,e_n)}\left(\mathbb{H}^1_{\omega_{1}} \times \cdots \times \mathbb{H}^1_{\omega_{n}}\right) \cong T_{e_1}\mathbb{H}^1_{\omega_{1}}\oplus\cdots\oplus T_{e_n}\mathbb{H}^1_{\omega_{n}}$ and $T_{\pi(e_1,\cdots,e_n)}\mathbb{H}^n_{\omega}$. Based on the explicit formula of the exponential map and $\pi$, we have
\begin{align*}
& \left(d\pi_{(e_1,\cdots,e_n)}\right)\left(\mathbf{a}_1,c_1,\cdots,\mathbf{a}_n,c_n\right)=\left(\mathbf{a}_1,\cdots,\mathbf{a}_n,\sum_{i=1}^nc_i\right)
\\
&
\left(\mathbf{a}_1,c_1,\cdots,\mathbf{a}_n,c_n\right)\in T_{(e_1,\cdots,e_n)}\left(\mathbb{H}^1_{\omega_{1}} \times \cdots \times \mathbb{H}^1_{\omega_{n}}\right).
\end{align*}
By \eqref{e.LieBracket} we see that the Lie brackets are preserved under $d\pi_{(e_1,\cdots,e_n)}$, and thus $d\pi_{(e_1,\cdots,e_n)}$ is a Lie algebra homomorphism.

At the identity $(e_1,\cdots,e_n)$, we have $\mathcal{H}_{(e_1,\cdots,e_n)}=\mathcal{H}^{\omega}_e=\mathbb{R}^{2n}$ and the differential of $\pi$ at $(e_1,\cdots,e_n)$ restricted to $\mathcal{H}_{(e_1,\cdots,e_n)}$ is the identity map on $\mathbb{R}^{2n}$. Note that $\langle \cdot,\cdot\rangle_{{\mathcal{H}}_{(e_1,\cdots,e_n)}}$ and $\langle \cdot,\cdot\rangle^{\omega}_{\mathcal{H}_e}$ are the same Euclidean inner products, so
\[
\left.d\pi_{(e_1,\cdots,e_n)}\right|_{\mathcal{H}_{(e_1,\cdots,e_n)}}:\mathcal{H}_{(e_1,\cdots,e_n)} \longrightarrow \mathcal{H}^{\omega}_{e}
\]
is an isometry.

Now we can use the fact that $\pi$ commutes with the left multiplication by Equation \eqref{pi1} to extend this to fibers of the horizontal distribution at any $g \in \mathbb{H}^1_{\omega_1}\times\cdots\times\mathbb{H}^1_{\omega_n}$. Namely, the left-invariance of sub-Riemannian metrics on these groups and \eqref{pi1} imply that
\[
\left.d\pi_{(g_1,\cdots,g_n)}\right|_{\mathcal{H}_{(g_1,\cdots,g_n)}}: \mathcal{H}_{(g_1,\cdots,g_n)} \longrightarrow \mathcal{H}^{\omega}_{\pi(g_1,\cdots,g_n)}
\]
is an isometry on the fibers of these sub-Riemannian distributions.

For any $(g_1,\cdots,g_n)=\left(x_1,y_1,z_1,\cdots,x_n,y_n,z_n\right) \in \mathbb{H}^1_{\omega_{1}} \times \cdots \times \mathbb{H}^1_{\omega_{n}}$ and $i=1, \cdots, n$ consider vector fields

\begin{align*}
& X^{\omega_{i}}\left( g_1, \cdots, g_n \right)=\frac{\partial}{\partial x_i}-\frac{\alpha_i}{2} y_i\frac{\partial}{\partial z_i},
\\
& Y^{\omega_{i}}(g_1, \cdots, g_n)=\frac{\partial}{\partial y_i}+\frac{\alpha_i}{2} x_i\frac{\partial}{\partial z_i}.
\end{align*}
Under the product sub-Riemannian structure $\{X^{\omega_{i}}, Y^{\omega_{i}}:i=1\cdots,n\}$ form an orthonormal frame for $\left(\mathbb{H}^1_{\omega_{1}} \times \cdots \times \mathbb{H}^1_{\omega_{n}}, \mathcal{H}, \langle \cdot,\cdot\rangle_{\mathcal{H}}\right)$. We see that the orthonormal frames   $\{X^{\omega_{i}}, Y^{\omega_{i}}:i=1,\cdots,n\}$ on the product manifold $\left( \mathbb{H}^1_{\omega_{1}} \times \cdots \times \mathbb{H}^1_{\omega_{n}}, \mathcal{H}, \langle \cdot,\cdot\rangle_{\mathcal{H}}\right)$ and $\{X_{i}^{\omega},Y_{i}^{\omega}:i=1,\cdots,n\}$ on $\left( \mathbb{H}^{n}_{\omega}, \mathcal{H}^{\omega}, \langle \cdot,\cdot\rangle^{\omega}_{\mathcal{H}} \right)$ defined by \eqref{e.CanonicalBasis} satisfy
\begin{align}
& (d\pi_{(g_{1},\cdots,g_{n})})\left(X^{\omega_i}(g_{1},\cdots,g_{n})\right)=X_i^{\omega}\left(\pi(g_{1},\cdots,g_{n})\right), \label{eqn.VectorRelation(2)1}
\\
&
(d\pi_{(g_{1},\cdots,g_{n})})\left(Y^{\omega_i}(g_{1},\cdots,g_{n})\right)=Y_i^{\omega}\left(\pi(g_{1},\cdots,g_{n})\right). \label{eqn.VectorRelation(2)2}
\end{align}
This shows that for any $f\in C^{\infty}_c\left(\mathbb{H}^{n}_{\omega}\right)$
\begin{align*}
& \left(\nabla_{\mathcal{H}}^{\omega}f\right) \left(\pi(g_{1},\cdots,g_{n})\right)
\\
& = \sum_{i=1}^{n}\left[
\left(
\left( X_i^{\omega}f \right)
X_i^{\omega}\right)\left(\pi\left( g_{1},\cdots,g_{n} \right) \right)+
\left(
\left( Y_i^{\omega}f \right)
Y_i^{\omega}\right)\left(\pi\left( g_{1},\cdots,g_{n} \right) \right)\right]
\\
&
=(d\pi_{(g_{1},\cdots,g_{n})})\left(\nabla_{\mathcal{H}}(f\circ \pi))(g_{1},\cdots,g_{n})\right),
\end{align*}
which proves \eqref{eqn.GradientLengthRelation(2)}  since $d\pi_{(g_1,\cdots,g_n)}$ restricted to $\mathcal{H}_{(g_1,\cdots, g_n)}$ is an isometry.

Finally, to show that the pushforward of the heat kernel measure $\mu_t$ by $\pi$ is the heat kernel measure $\mu_t^{\omega}$ on $\mathbb{H}^n_{\omega}$ we will use Proposition~\ref{prop.HeatKernelMeasureDF2}. For any $f\in C^{\infty}_c\left(\mathbb{H}^{n}_{\omega}\right)$, we have $f\circ\pi \in C^{\infty}_c\left(\mathbb{H}^1_{\omega_{1}} \times \cdots \times \mathbb{H}^1_{\omega_{n}}\right)$, and therefore by \eqref{eqn.VectorRelation(2)1} and \eqref{eqn.VectorRelation(2)2}
\begin{align}
\left(\sum_{i=1}^n\Delta^{\omega_i}_{\mathcal{H}}\right)\left(f \circ \pi\right)=\left(\Delta^{\omega}_{\mathcal{H}} f\right) \circ \pi. \label{eqn.LaplacianRelation(2)}
\end{align}
By Proposition~\ref{prop.HeatKernelMeasureDF2} applied to the heat kernel measure $\mu_t$ we see that
\begin{align*}
& \frac{d}{dt} \int_{\mathbb{H}^1_{\omega_{1}} \times \cdots \times \mathbb{H}^1_{\omega_{n}}} (f\circ \pi)\left(g_{1},\cdots,g_{n}\right) d\mu_t
\\
& =\int_{\mathbb{H}^1_{\omega_{1}} \times \cdots \times \mathbb{H}^1_{\omega_{n}}} \left(\left(\frac{1}{2}\sum_{i=1}^n\Delta^{\omega_i}_{\mathcal{H}}\right) (f\circ \pi)\right)\left(g_{1},\cdots,g_{n}\right) d\mu_t,
\\
&
\lim_{t \to 0} \int_{\mathbb{H}^1_{\omega_{1}} \times \cdots \times \mathbb{H}^1_{\omega_{n}}} (f \circ \pi)\left(g_{1},\cdots,g_{n}\right) d\mu_t=(f\circ\pi)(e_{1},\cdots,e_{n}).
\end{align*}
Then by Equation~\ref{eqn.LaplacianRelation(2)} and using the change of variables in the heat equation for $\mu_t$ together with Proposition~\ref{prop.HeatKernelMeasureDF2} applied to $\pi_{\#}\mu_t$ implies that $\pi_{\#}\mu_t$ is the heat kernel measure $\mu_t^{\omega}$ on $\mathbb{H}^n_{\omega}$.
\end{proof}

\begin{remark}
From the proof of Proposition~\ref{p.QuotientLemma}, we see that we can identify $\mathcal{H}^{\omega}_g \cong  \mathcal{H}^{\omega_1}_{g_1}\oplus\cdots\oplus\mathcal{H}^{\omega_n}_{g_n}$ for $g\in\mathbb{H}^n_{\omega}$ and $(g_1,\cdots,g_n)\in\mathbb{H}^1_{\omega_{1}} \times \cdots \times \mathbb{H}^1_{\omega_{n}}$ with $g=\pi(g_1,\cdots,g_n)$.
\end{remark}

\begin{remark}
The results in \cite[p. 38]{Fraser2001a} say that we can identify the isotropic Heisenberg group $\mathbb{H}^n_{\omega_0}$ with a quotient group. By Proposition~\ref{p.QuotientLemma}, we can also identify $\mathbb{H}^n_{\omega}$ with a quotient group $G/N$ where $G=\mathbb{H}^1_{\omega_{1}} \times \cdots \times \mathbb{H}^1_{\omega_{n}}$ and $N=\{(\mathbf{0},z_1,\cdots,\mathbf{0},z_n)\in\mathbb{H}^1_{\omega_{1}} \times \cdots \times \mathbb{H}^1_{\omega_{n}}:\sum_{i=1}^nz_i=0\}$
\end{remark}

\begin{theorem}\label{thm.LSINonisotropic}
The logarithmic Sobolev inequality $\operatorname{LSI}\left( C\left(\omega\right)t, \mu_t^{\omega} \right)$ holds on $\mathbb{H}^{n}_{\omega}$, where $C\left(\omega\right)$ can be chosen to be equal to  $C\left(\omega_{0}\right)$. In particular,   the logarithmic Sobolev constant $C\left(\omega\right)$ is independent of both the symplectic form $\omega$ and the dimension of the group $\mathbb{H}^{n}_{\omega}$.
\end{theorem}

\begin{proof}
For any $f\in C^{\infty}_c\left(\mathbb{H}^{n}_{\omega}\right)$, we have $f\circ \pi\in C^{\infty}_c\left(\mathbb{H}^1_{\omega_{1}} \times \cdots \times \mathbb{H}^1_{\omega_{n}}\right)$ where $\pi$ is defined by \eqref{pi}. Then by Proposition~\ref{prop.LSITensorization} we can apply the logarithmic Sobolev inequality $\operatorname{LSI}\left(C\left(\omega_0\right)t,\mu_t\right)$ on $\mathbb{H}^1_{\omega_{1}} \times \cdots \times \mathbb{H}^1_{\omega_{n}}$ to $f\circ \pi$ to see that
\begin{align*}
& \int_{\mathbb{H}^1_{\omega_{1}} \times \cdots \times \mathbb{H}^1_{\omega_{n}}}(f\circ \pi)^2\log (f\circ \pi)^2 d\mu_t
\\
& -\left(\int_{\mathbb{H}^1_{\omega_{1}} \times \cdots \times \mathbb{H}^1_{\omega_{n}}}(f\circ \pi)^2d\mu_t\right) \log \left(\int_{\mathbb{H}^1_{\omega_{1}} \times \cdots \times \mathbb{H}^1_{\omega_{n}}}(f\circ \pi)^2d\mu_t\right)
\\
&
\leqslant C\left(\omega_0\right)t \int_{\mathbb{H}^1_{\omega_{1}} \times \cdots \times \mathbb{H}^1_{\omega_{n}}} \vert \nabla_{\mathcal{H}} (f\circ \pi)\vert^2_{\mathcal{H}} d\mu_t.
\end{align*}
Using the change of variable $\pi\left(g_1,\cdots,g_n\right)\longmapsto g$ in this inequality together with Proposition~\ref{p.QuotientLemma}, we see that $\operatorname{LSI}\left(C\left(\omega_0\right)t,\mu_t^{\omega}\right)$ holds on $\mathbb{H}^n_{\omega}$. Thus  we can take $C\left(\omega\right)=C\left(\omega_0\right)$ which is independent of $\omega$, $n$ and the dimension of the group $\mathbb{H}^{n}_{\omega}$.
\end{proof}

\section{The second approach: tensorization and lifting reversed} \label{sec.OrderReversed}
This section describes an approach where the order of tensorization and the lifting to the product group is reversed. Define
\begin{align}\label{pi_omega}
& \pi_{\omega}:\mathbb{H}^1_{\omega_{0}} \times \cdots \times \mathbb{H}^1_{\omega_{0}} \rightarrow \mathbb{H}^{n}_{\omega}, \notag
\\
& \pi_{\omega}(g_1,\cdots,g_n):=\pi_{\omega}(x_{1}, y_{1},z_{1},\cdots,x_{n},y_{n},z_{n})=(x_{1},y_{1},\cdots,x_{n},y_{n}, z),
\\
&
z=\sum_{i=1}^{n} \alpha_iz_i \notag
\end{align}
for any $(g_1,\cdots,g_n)\in \mathbb{H}^1_{\omega_{0}} \times \cdots \times \mathbb{H}^1_{\omega_{0}}$, where $\alpha_i$ are given by Equation~\ref{e.SymplForm} for $i=1,\cdots,n$. Note that in this approach the lifting depends on the symplectic form $\omega$. The next statement shows that we can still view $\pi_{\omega}$ as an homomorphism between sub-Riemannian manifolds.

\begin{proposition}[$\mathbb{H}^{n}_{\omega}$ and $\mathbb{H}^1_{\omega_0} \times \cdots \times \mathbb{H}^1_{\omega_0}$] \label{p.QuotientLemma(2)}
The map $\pi_{\omega}$ is a Lie group homomorphism such that for any $(g_1,\cdots,g_n)\in\mathbb{H}^1_{\omega_{0}} \times \cdots \times \mathbb{H}^1_{\omega_{0}}$, it commutes with the left-translation $L_{(g_1,\cdots,g_n)}$, i.e.
\begin{align} \label{pi2}
\pi_{\omega}\circ L_{(g_1,\cdots,g_n)}=L_{\pi_{\omega}(g_1,\cdots,g_n)} \circ \pi_{\omega},
\end{align}
and the differential of $\pi_{\omega}$ at $(g_1,\cdots,g_n)$ restricted to horizontal spaces, $d\left(\pi_{\omega}\right)_{(g_1,\cdots,g_n)}\vert_{\mathcal{H}_{(g_1,\cdots,g_n)}}: \mathcal{H}_{(g_1,\cdots,g_n)} \rightarrow \mathcal{H}^{\omega}_{\pi_{\omega}(g_1,\cdots,g_n)}$ is an isometry. Moreover, for any $f\in C^{\infty}_c\left(\mathbb{H}^{n}_{\omega}\right)$
\begin{align} \label{eqn.GradientLengthRelation(3)}
\vert \nabla_{\mathcal{H}}(f\circ \pi_{\omega})\vert_{\mathcal{H}}=\vert \nabla^{\omega}_{\mathcal{H}}f\vert_{\mathcal{H}^{\omega}} \circ \pi_{\omega}.
\end{align}
In addition, the pushforward of the heat kernel measure $\mu_t$ by $\pi_{\omega}$ is the heat kernel measure $\mu_t^{\omega}$ on $\mathbb{H}^n_{\omega}$.
\end{proposition}

\begin{proof}
In this case, the explicit formula for the differential of $\pi_{\omega}$ at $(e_1,\cdots,e_n)$,  $d(\pi_{\omega})_{(e_1,\cdots,e_n)}: T_{(e_1,\cdots,e_n)} \left(\mathbb{H}^1_{\omega_0} \times \cdots \times \mathbb{H}^1_{\omega_0}\right) \rightarrow T_{\pi_{\omega}(e_1,\cdots,e_n)}\mathbb{H}^n_{\omega}$ is
\begin{align*}
& \left(d\left(\pi_{\omega}\right)_{(e_1,\cdots,e_n)}\right)\left(\mathbf{a}_1,c_1,\cdots,\mathbf{a}_n,c_n\right)=\left(\mathbf{a}_1,\cdots,\mathbf{a}_n,\sum_{i=1}^n\alpha_ic_i\right),
\\
&
\left(\mathbf{a}_1,c_1,\cdots,\mathbf{a}_n,c_n\right)\in T_{(e_1,\cdots,e_n)} \left(\mathbb{H}^1_{\omega_0} \times \cdots \times \mathbb{H}^1_{\omega_0}\right) \cong T_{e_1}\mathbb{H}^1_{\omega_0}\oplus\cdots\oplus T_{e_n}\mathbb{H}^1_{\omega_0}.
\end{align*}
The rest of the proof is similar to the proof of Proposition~\ref{p.QuotientLemma}.
\end{proof}
Using the lifting $\pi_{\omega}$, we can also prove Theorem~\ref{thm.LSINonisotropic} as follows.

\begin{proof}[Second proof of Theorem~\ref{thm.LSINonisotropic}]
First we can apply Proposition~\ref{prop.LSITensorization} to the group $\mathbb{H}^1_{\omega_{0}} \times \cdots \times \mathbb{H}^1_{\omega_{0}}$ to see that it satisfies a logarithmic Sobolev inequality $\operatorname{LSI}\left( C\left(\omega_0, n, t\right),
\mu_t^{\omega_{0}}\otimes \cdots \otimes \mu_t^{\omega_{0}}\right)$,
where the logarithmic Sobolev constant can be chosen to be $C\left(\omega_0, n, t\right)=C\left(\omega_0\right)t$. For any $f\in C^{\infty}_c\left(\mathbb{H}^{n}_{\omega}\right)$, we have $f\circ \pi_{\omega}\in C^{\infty}_c\left(\mathbb{H}^1_{\omega_{0}} \times \cdots \times \mathbb{H}^1_{\omega_{0}}\right)$ where $\pi_{\omega}$ is defined by \eqref{pi_omega}. As in the first proof of Theorem~\ref{thm.LSINonisotropic}, we can  use the change of variables $\pi_{\omega}\left(g_1,\cdots,g_n\right)\mapsto g$ in the logarithmic Sobolev inequality on $\mathbb{H}^1_{\omega_{0}} \times \cdots \times \mathbb{H}^1_{\omega_{0}}$ for $f\circ \pi_{\omega}$, and together with Proposition~\ref{p.QuotientLemma(2)} we get the same result.
\end{proof}

\section{Logarithmic Sobolev inequalities on infinite-dimensional Heisenberg groups} \label{sec.LSIInfinite}

In this section, we consider an application of the results on non-isotropic Heisenberg groups to infinite-dimensional Heisenberg groups with a one-dimensional center. We aim to prove the logarithmic Sobolev inequality on such an infinite-dimensional Heisenberg group by the finite-dimensional projection approximation approach used in \cite{DriverGordina2008, BaudoinGordinaMelcher2013}. That is, we will approximate the logarithmic Sobolev inequality on the infinite-dimensional Heisenberg group by logarithmic Sobolev inequalities on finite-dimensional projection groups which are non-isotropic Heisenberg groups discussed in previous sections. The crucial ingredient here is that we proved previously that the LSI constant can be chosen to be independent of the dimension of finite-dimensional projection groups. 

We start by reviewing the definitions for infinite-dimensional
Heisenberg-like groups, which are infinite-dimensional Lie groups modelled on an
abstract Wiener space, and collect some properties of the finite-dimensional projection approximation. We may omit some details, but much of the material in this section also appears in \cite{DriverGordina2008, GordinaMelcher2013}, and subsequently in \cite{BaudoinGordinaMelcher2013, DriverEldredgeMelcher2016, Gordina2017}.

\subsection{Abstract Wiener spaces}

\label{s.wiener}
We start by summarizing several well-known properties of Gaussian measures and abstract Wiener spaces that are needed later.  These results as well as more details on abstract Wiener spaces may be found in \cite{BogachevGaussianMeasures, KuoLNM1975}.

Suppose that $W$ is a real separable Banach space and $\mathcal{B}_{W}$ is
the Borel $\sigma$-algebra on $W$.

\begin{definition}
\label{d.2.1}
A measure $\mu$ on $(W,\mathcal{B}_{W})$ is called a (mean zero,
non-degenerate) {\it Gaussian measure} provided that its characteristic
functional is given by
\begin{equation}
\label{e.gauss}
\hat{\mu}(u) := \int_W e^{iu(x)} d\mu(x)
	= e^{-\frac{1}{2}q(u,u)}, \qquad \text{ for all } u \in W^*,
\end{equation}
for $q=q_\mu:W^*\times W^*\rightarrow\mathbb{R}$ a symmetric, positive definite quadratic form.
That is, $q$ is a real inner product on $W^*$.
\end{definition}

A proof of the following standard theorem may be found for example in  \cite[Appendix
A]{DriverGordina2008} and \cite[Lemma 3.2]{BaudoinGordinaMelcher2013}.

\begin{theorem}
\label{t.2.3}
Let $\mu$ be a Gaussian measure on a real separable Banach space $W$.
For $p\in[1,\infty)$, let
\begin{equation}
\label{e.2.2}
C_p :=\int_W \|w\| _{W}^{p} \,d\mu(w).
\end{equation}
For $w\in W$, let
\[
\|w\|_H := \sup\limits_{u\in W^*\setminus\{0\}}\frac{|u(w)|}{\sqrt{q(u,u)}}
\]
and define the {\em Cameron-Martin subspace} $H\subset W$ by
\[ H := \{h\in W : \|h\|_H < \infty\}. \]
Then
\begin{enumerate}
\item \label{i.1}
For all  $p\in[1,\infty)$, $C_p<\infty$.

\item $H$ is a dense subspace of $W$.

\item There exists a unique inner product $\langle\cdot,\cdot\rangle_H$
on $H$ such that $\|h\|_H^2 = \langle h,h\rangle_H$ for all $h\in H$, and
$H$ is a separable Hilbert space with respect to this inner product.

\item \label{i.3}
For any $h\in H$,
$\|h\|_W \le \sqrt{C_2} \|h\|_H$.

\item \label{i.5}
If $\{e_j\}_{j=1}^\infty$ is an orthonormal basis for $H$, then for any
$u,v\in H^*$
\[ q(u,v) = \langle u,v\rangle_{H^*} = \sum_{j=1}^\infty u(e_j)v(e_j).
\]
\item \label{l.q}
If $u,v\in W^*$, then
\[
\int_W u(w)v(w)\,d\mu(w) = q(u,v).
 \]
\end{enumerate}
\end{theorem}

It follows from \eqref{i.3} that any $u\in W^*$ restricted to $H$ is in
$H^*$.  Therefore, by \eqref{i.5} and  \eqref{l.q}
\begin{equation}\label{e.bl}
\int_W u^2(w)\,d\mu(w) = q(u,u) = \|u\|_{H^*}^2 = \sum_{j=1}^\infty |u(e_j)|^2.
\end{equation}

\subsection{Infinite-dimensional Heisenberg-like groups}\label{s.3}

We revisit the definition of the infinite-dimensional Heisenberg-like groups that were first considered in \cite{DriverGordina2008}. Note that since we are interested in hypoelliptic heat kernel measures on these groups, we consider the topology described previously in \cite{GordinaMelcher2013, BaudoinGordinaMelcher2013, DriverEldredgeMelcher2016}.  First we set the following notation which will be used for the rest of the paper. Note that we consider only the case of the one-dimensional center.

\begin{notation}
Let $(W,H,\mu)$ be a real abstract Wiener space.   Suppose $\omega: W \times W \rightarrow \mathbb{R}$ is a continuous skew-symmetric bilinear form on $W$.
\end{notation}

\begin{remark}
As stated in \cite[Proposition 3.14]{DriverGordina2008} it is suprising to see that the continuity of the symplectic form $\omega$ implies that
\begin{align*}
\left\| \omega\right\| _{2}^{2}
	:= \left\| \omega\right\| _{H^{\ast}\otimes H^{\ast}\otimes\mathbb{R}}^2
	:= \sum_{i,j=1}^{\infty}\left\vert \omega\left(  e_{i},e_{j}\right)  \right\vert^{2}< \infty,
\end{align*}
where $\{e_i\}_{i=1}^{\infty}$ is an orthonormal basis for $H$, and thus the Hilbert-Schmidt norm of $\omega$ is finite.
\end{remark}

\begin{definition} \label{df.InfiniteDimensionalHeisenbergGroup}
Let $\mathfrak{g}$ denote $W\times\mathbb{R}$ when thought of as a Lie algebra
with the Lie bracket given by
\begin{equation}
\label{e.6.4}
[(X_1,V_1), (X_2,V_2)] := (0, \omega(X_1,X_2)).
\end{equation}
Let $G$ denote $W\times\mathbb{R}$ when thought of as a group with
multiplication given by
\begin{equation*}
 g_1 g_2 := g_1 + g_2 + \frac{1}{2}[g_1,g_2],
\end{equation*}
where $g_1$ and $g_2$ are viewed as elements of $\mathfrak{g}$. For $g_i=(w_i,c_i)$, this may be written equivalently as
\begin{equation}
\label{e.3.2}
(w_1,c_1)\cdot(w_2,c_2) = \left( w_1 + w_2, c_1 + c_2 +
    \frac{1}{2}\omega(w_1,w_2)\right).
\end{equation}
We will call $G$ constructed in this way an infinite-dimensional \emph{Heisenberg group}.
\end{definition}
It is easy to verify that, given this bracket and multiplication,
$\mathfrak{g}$ is indeed a Lie algebra and $G$ is a group.
Note that $g^{-1}=-g$ and the identity $e=(0,0)$.

\begin{notation}
Let $\mathfrak{g}_{CM}$ denote $H\times\mathbb{R}$ when thought of as a Lie
subalgebra of $\mathfrak{g}$, and we will refer to $\mathfrak{g}_{CM}$ as the
{\em Cameron-Martin subalgebra} of $\mathfrak{g}$. Similarly, let $G_{CM}$
denote $H\times\mathbb{R}$ when thought of as a subgroup of $G$, and we will
refer to $G_{CM}$ as the {\em Cameron-Martin subgroup} of $G$.
\end{notation}
We will equip $\mathfrak{g}=G$ with the homogeneous norm
\[ \|(w,c)\|_{\mathfrak{g}} := \sqrt{\|w\|_W^2 + \vert c\vert}, \]
and analogously on $\mathfrak{g}_{CM}=G_{CM}$ we define
\[ \|(A,a)\|_{\mathfrak{g}_{CM}} := \sqrt{\|A\|_H^2 + \vert a\vert}. \]

One may easily see that $G$ and $G_{CM}$ are topological groups with respect
to the topologies induced by the homogeneous norms, see for example \cite[Lemma 2.9]{GordinaMelcher2013}.

Before proceeding, we describe the basic examples for the construction of
these infinite-dimensional Heisenberg groups.

\begin{example} [Finite-dimensional non-isotropic Heisenberg group]
\label{ex.Heis}
Let $W=H\cong\mathbb{R}^{2n}$. Suppose $\omega$ is a symplectic form on $\mathbb{R}^{2n}$. Then $G=\mathbb{R}^{2n}\times\mathbb{R}$ equipped with the group operation defined by \eqref{e.3.2} is a non-isotropic
Heisenberg group with the group law defined by \eqref{GroupLaw}.
\end{example}

\begin{example}[Heisenberg group of a symplectic vector space]
\label{ex.infHeis}
Let $(K,\langle\cdot,\cdot\rangle)$ be a Hilbert space and $Q$ be a
strictly positive trace class operator on $K$.  For $h, k\in K$, let $\langle
h, k\rangle_Q:= \langle h, Qk \rangle$ and $\|h\|_Q:= \sqrt{\langle
h, h\rangle_Q}$, and let $(K_Q,\langle\cdot, \cdot\rangle_Q)$ denote the Hilbert
space completion of $(K,\|\cdot\|_Q)$.
Then $W=(K_Q)_{\operatorname{Re}}$ and $H=K_{\operatorname{Re}}$ determines an abstract Wiener
space (see, for example,  of \cite[Exercise 17 on p.59]{KuoLNM1975}).  Letting
\[
\omega( w,z ) := \operatorname{Im}\langle w, z \rangle_Q,
\]
then $G=(K_Q)_{\operatorname{Re}}\times\mathbb{R}$ equipped with a group operation as
defined by \eqref{e.3.2} is an infinite-dimensional Heisenberg-like group.
\end{example}

\subsection{Finite-dimensional projection groups}\label{s.gpproj}
The finite-dimensional projections of $G$ defined in this section will be important in the sequel.  The construction of these projections is quite natural as they come from the usual projections of the abstract Wiener space; however, the projections defined here are not group homomorphisms.

As before, let $(W,H,\mu)$ denote an abstract Wiener space.
Let $i:H\rightarrow W$ be the inclusion map, and $i^*:W^*\rightarrow H^*$ be
its transpose so that $i^*\ell:=\ell\circ i$ for all $\ell\in W^*$.  Also,
let
\[ H_* := \{h\in H: \langle\cdot,h\rangle_H\in \mathrm{Range}(i^*)\subset H^*\}.
\]
That is, for $h\in H$, $h\in H_*$ if and only if $\langle\cdot,h\rangle_H\in
H^*$ extends to a continuous linear functional on $W$, which we will continue
to denote by $\langle\cdot,h\rangle_H$.  Because $H$ is a dense subspace of
$W$, $i^*$ is injective and thus has a dense range.  Since
$H\ni h\mapsto\langle\cdot,h\rangle_H\in H^*$ is a
linear isometric isomorphism, it follows that $H_*\ni
h\mapsto\langle\cdot,h\rangle_H\in W^*$ is a linear isomorphism
also, and so $H_*$ is a dense subspace of $H$.

Suppose that $P:H\rightarrow H$ is a finite rank orthogonal projection
such that $PH\subset H_*$.  Let $\{ e_j\}_{j=1}^n$ be an orthonormal basis for
$PH$.  Then we may extend $P$ to a (unique) continuous operator
from $W\rightarrow H$ (still denoted by $P$) by letting
\begin{equation}
\label{e.proj}
Pw := \sum_{j=1}^n \langle w, e_j\rangle_H  e_j
\end{equation}
for all $w\in W$.

\begin{notation}
\label{n.proj}
Let $\mathrm{Proj}(W)$ denote the collection of finite rank projections
on $W$ such that
\begin{enumerate}
\item $PW\subset H_*$,
\item $P|_H:H\rightarrow H$ is an orthogonal projection (that is, $P$ has the form given in equation
\eqref{e.proj}), and
\item $PW$ is sufficiently large to satisfy
H\"ormander's condition (that is, $\{\omega(A,B):A,B\in PW\}=\mathbb{R}$).
\end{enumerate}
For each $P\in\mathrm{Proj}(W)$, we define $G_P:=
PW\times\mathbb{R}\subset H_*\times\mathbb{R}$ and a corresponding
projection $\pi_P:G\rightarrow G_P$ \[ \pi_P(w,x):= (Pw,x). \] We
will also let $\mathfrak{g}_P=\mathrm{Lie}(G_P) = PW\times\mathbb{R}$.
\end{notation}

Note that for each $P\in\mathrm{Proj}(W)$, $G_P$
is a finite-dimensional connected unimodular Lie group, and  $\mathfrak{g}_P$ is step
2 stratified Lie algebra with  $\mathcal{H}=PH$ and $\mathcal{V}=\mathbb{R}$. Moreover, when $\omega$ is restricted to $PW\times PW$, we see that $\omega\vert_{PW\times PW}:PW\times PW \rightarrow \mathbb{R}$ is a symplectic form from the non-degeneracy and the skew-symmetry of $\omega$. By Theorem~\ref{t.SymplBasis} we have $\operatorname{dim}PW$ is even. Together with Proposition~\ref{p.SymplBasis} and \eqref{e.3.2}, we see that for any each $P\in\mathrm{Proj}(W)$, $G_P$ is a non-isotropic Heisenberg group equipped with the group law given by
\begin{align*}
(v_1,c_1)\cdot(v_2,c_2) & = \left( v_1 + v_2, c_1 + c_2 +\frac{1}{2}\omega(v_1,v_2)\right)
\\
&
=\left( v_1 + v_2, c_1 + c_2 +\frac{1}{2}\omega\vert_{PW\times PW}(v_1,v_2)\right) \end{align*}
for any $(v_i,c_i)=(Pw_i,c_i)\in PW\times \mathbb{R}$, which is consistent with \ref{GroupLaw}.

\subsection{Subelliptic Laplacian and the heat kernel measure on $G$}

\subsubsection{Subelliptic Laplacian and horizontal gradient on $G$}
\label{s.deriv}

In this section, we give the definition of the subelliptic Laplacian and the horizontal gradient on $G$ analogously to how it is done in the non-isotropic case. To begin with, we recall some definitions of derivatives on $G$. For more details, we refer to \cite[p. 8-10]{DriverGordina2008} and \cite[Section 3.4]{BaudoinGordinaMelcher2013}.

For $x\in G$ we denote by $L_x:G\rightarrow G$ the left multiplication by $x$.  As $G$
is a vector space, to each $x\in G$ we can associate the tangent
space $T_x G$ to $G$ at $x$, which is naturally isomorphic to $G$.

\begin{notation}[Linear and group derivatives]
\label{n.3.5}
Let $f:G\rightarrow\mathbb{C}$ denote a Fr\'{e}chet smooth function for $G$
considered as a Banach space with respect to the norm
\[
|(w,c)|_G := \sqrt{\|w\|_W^2+\vert c\vert^2}.
\]
Then, for $x\in G$, and $h,k\in\mathfrak{g}$, let
\[ f'(x)h := \partial_h f(x) = \frac{d}{dt}\bigg|_0f(x+th) \]
and
\[
f''(x)  \left(  h\otimes k\right)  :=\partial
_{h}\partial_{k}f(x).
\]
For $v,x\in G$, let $v_x \in T_x G$ denote the tangent vector
satisfying $v_xf=f'(x)v$.  If $x(t)$ is any smooth curve in
$G$ such that $x(0) = x$ and $\dot{x}(0)=v$ (for example,
$x(t) = x+tv$), then
\[ L_{g*} v_x = \frac{d}{dt}\bigg|_0 g\cdot x(t). \]
In particular, for $x=e$ and $v_e=h\in\mathfrak{g}$, again we let
$\tilde{h}(g):=L_{g*}h$, so that $\tilde{h}$ is the unique left invariant
vector field on $G$ such that $\tilde{h}(e)=h$.  As usual we view
$\tilde{h}$ as a first order differential operator acting on smooth
functions by
\[ (\tilde{h}f)(x) = \frac{d}{dt}\bigg|_0 f(x\cdot \sigma(t)), \]
where $\sigma(t)$ is a smooth curve in $G$ such that $\sigma(0)=e$ and
$\dot{\sigma}(0)=h$ (for example, $\sigma(t)=th$).
\end{notation}

The explicit formula to compute $\tilde{h}f$ is given in \cite[Proposition 3.7]{DriverGordina2008}. Moreover, \cite[Proposition 3.7]{DriverGordina2008} shows that the Lie algebra structure on $\mathfrak{g}$ induced by the Lie
algebra structure on the left invariant vector fields on $G$ is the same as the Lie algebra structure defined by \eqref{e.6.4}, which is consistent with the finite-dimensional setting.

Now we recall the definition of some special class of functions that are used often in this setting.

\begin{definition}
\label{d.cyl}
A function $f:G\rightarrow\mathbb{C}$ is a {\em (smooth) cylinder function}
if it may be written as $f=F\circ\pi_P$, for some $P\in\mathrm{Proj}(W)$
and (smooth) function $F:G_P\rightarrow\mathbb{C}$. A \emph{cylinder polynomial} is a cylinder function, $f=F\circ \pi_{P}:G \rightarrow \mathbb{C}$, where $P\in\operatorname*{Proj}(W)$ and $F$ is a real or complex polynomial function on $G_P$.
\end{definition}

We consider the second-order differential operator below as an analogue of the sub-Laplacian in the finite-dimensional setting.

\begin{definition}
Let $\left\{  e_{j}\right\}_{j=1}^{\infty}$ be an orthonormal
basis for $H$. For any smooth cylinder function $f:G\rightarrow
\mathbb{R}$, we define the \emph{subelliptic Laplacian} as
\begin{align} \label{SubellipticLaplacian}
Lf(x)  :=\sum_{j=1}^{\infty}\left[  \widetilde{\left(
e_{j},0\right)  }^{2}f\right]  (x).
\end{align}
\end{definition}

By \cite[Proposition 3.17]{BaudoinGordinaMelcher2013}, \eqref{SubellipticLaplacian} is well-defined and independent of the choice of basis.

\begin{definition}
For any cylinder polynomial $u$, define the \emph{horizontal gradient} $\operatorname{grad}_H:G \rightarrow H$ of $u$ by
\begin{align} \label{eqn.HorizaontalGradientInfinite}
\langle \operatorname{grad}_Hu,h\rangle_H=\widetilde{\left(h,0\right)}u
\end{align}
for any $h\in H$.
\end{definition}

Let $\left\{  e_{j}\right\}_{j=1}^{\infty}$ be an orthonormal
basis for $H$. Then we have
\begin{align*}
\operatorname{grad}_Hu=\sum_{j=1}^{\infty}\left(  \widetilde{\left(
		e_{j},0\right)  }u\right) \left(  x\right).
\end{align*}

For the finite-dimensional groups $G_P$ we may define the same
operators $L_Pf$ and $\operatorname{grad}_H^Pf$ for $f\in C^\infty(G_P)$.  In particular, if $\{e_i\}_{i=1}^n$ is an orthonormal basis of $PH$, then
\[ L_Pf = \sum_{j=1}^n \widetilde{(e_j,0)}^2f
	\quad\text{ and }\quad
\operatorname{grad}_H^Pf = \sum_{j=1}^n
	\widetilde{(e_j,0)}f \]
which are consistent with \eqref{df.SubLaplacian} and Definition~\ref{df.HorizontalGradient}.

\subsubsection{Distances on $G_{CM}$}
\label{s.length}

The sub-Riemannian distance on $G_{CM}$ can be defined similarly to how it is done in finite dimensions.  We recall its definition and relevant properties, including the fact that the topology induced by this metric is equivalent to the topology induced by $\|\cdot\|_{\mathfrak{g}_{CM}}$. We do not use these facts, but we include them for completeness.

\begin{notation}[Horizontal distance on $G_{CM}$]
\label{n.length}

\begin{enumerate}

\item For $x=(A,a)\in G_{CM}$, let
\[
|x|_{\mathfrak{g}_{CM}}^2 : = \|A\|_H^2 + \vert a\vert^2.
\]
The \emph{length} of a $C^1$-path $\sigma:[a,b]\rightarrow
G_{CM}$ is defined as
\[
\ell(\sigma)
	:= \int_a^b |L_{\sigma^{-1}(s)*}\dot{\sigma}(s)|_{\mathfrak{g}_{CM}} \,ds.
\]

\item \label{i.2}
A $C^1$-path $\sigma:[a,b]\rightarrow G_{CM}$ is {\em horizontal} if
$L_{\sigma(t)^{-1}*}\dot{\sigma}(t)\in H\times\{0\}$
for a.e.~$t$.  Let $C^{1,h}_{CM}$ denote the set of horizontal paths
$\sigma:[0,1]\rightarrow G_{CM}$.

\item The {\em horizontal distance} between $x,y\in G_{CM}$ is defined by
\[ d(x,y) := \inf\{\ell(\sigma): \sigma\in C^{1,h}_{CM} \text{ such
    that } \sigma(0)=x \text{ and } \sigma(1)=y \}. \]
\end{enumerate}
The horizontal distance is defined analogously on $G_P$ and will be denoted by
$d_P$.
\end{notation}

\begin{remark}
\label{r.horiz}
Note that if $\sigma(t)=(A(t),a(t))$ is a horizontal path, then
\[ L_{\sigma(t)^{-1}*}\dot{\sigma}(t)
	= \left(\dot{A}(t), \dot{a}(t) - \frac{1}{2}\omega(A(t),\dot{A}(t))\right)
		\in H\times\{0\}
\]
implies that $\sigma$ must satisfy
\[ a(t) = a(0) + \frac{1}{2}\int_0^t \omega(A(s),\dot{A}(s))\,ds, \]
and the length of $\sigma$ is given by
\begin{align*}
\ell(\sigma)
	= \int_0^1 |L_{\sigma^{-1}(s)*}\dot{\sigma}(s)|_{\mathfrak{g}_{CM}}\,ds
	= \int_0^1 \|\dot{A}(s)\|_H\,ds .
\end{align*}
\end{remark}

The following statement is \cite[Proposition 2.17, Proposition 2.18]{GordinaMelcher2013}.  We
refer the reader to that paper for the proofs.
\begin{proposition}[Proposition 2.17 and Proposition 2.18 in \cite{GordinaMelcher2013}]
\label{p.length}
If the symplectic form $\omega$ is a surjective map onto $\mathbb{R}$,
then there exist finite constants $K_1=K_1(\omega)$ and $K_2=K_2(\omega)$ such that

\[
K_1(\|A\|_H+\sqrt{\vert a\vert}) \leqslant d(e,(A,a))
	\leqslant K_2(\|A\|_H+\sqrt{\vert a\vert}),
\]
for all $(A,a)\in\mathfrak{g}_{CM}$.  In particular, the topologies induced by $d$ and $\|\cdot\|_{\mathfrak{g}_{CM}}$
are equivalent.
\end{proposition}

\begin{remark}
The equivalence of the homogeneous norm and horizontal distance
topologies is a standard result in finite dimensions.  However, the
usual proof of this result relies on compactness arguments that
must be avoided in infinite dimensions.  Thus, the proof for
Proposition \ref{p.length} included in \cite{GordinaMelcher2013}
necessarily relies on different methods particular to the structure
of the present groups. The reader is referred to \cite{GordinaMelcher2013}
for further details.

\end{remark}

As stated in \cite[Lemma 3.23]{BaudoinGordinaMelcher2013}, the horizontal distances on $G_{CM}$ and $G_P$ are connected as follows.

\begin{lemma} [Lemma 3.23 in \cite{BaudoinGordinaMelcher2013}]
\label{l.dn}
Let $\{P_n\}_{n=1}^\infty\subset\mathrm{Proj}(W)$ such that $P_n|_H\uparrow I_H$. For any $n\in \mathbb{N}$ and $x\in G_{P_n}$, then
\[
d_{P_n}\left( e, x \right) \rightarrow d\left( e, x \right), \text{ as }
n\rightarrow\infty. \]
\end{lemma}

\subsubsection{Heat kernel measure on $G$}\label{sec.HeatKernelMeasureInfinite}

In this section, we recall the definition of the heat kernel measure on $G$ and some relevant properties. For simplicity, we only include the key ideas and we refer to \cite[Section 5.1]{BaudoinGordinaMelcher2013}, \cite[Section 2.6]{GordinaMelcher2013} and \cite[Section 4, 8]{DriverEldredgeMelcher2016} for more details.

First, we recall the definition of hypoelliptic Brownian motion $\{g_t\}_{t\ge0}$
on $G$ and state its basic properties. Let $\{B_t\}_{t\ge0}$ be a Brownian motion on $W$ with variance determined by
\[
\mathbb{E}\left[\langle B_s,h\rangle_H \langle B_t,k\rangle_H\right]
    = \langle h,k \rangle_H \min(s,t),
\]
for all $s,t\ge0$ and $h,k\in H_*$. A \emph{hypoelliptic Brownian motion} on $G$ is the continuous $G$-valued process given by
\[g_t = \left( B_t, \frac{1}{2}\int_0^t \omega(B_s,dB_s)\right)
\]
where $\int_0^t \omega(B_s,dB_s)$ is taken to be the limiting process defined in \cite[Proposition 4.1, ]{DriverGordina2008} and its well-definedness relies on the finiteness of the Hilbert-Schmidt norm of $\omega$. By \cite[Proposition 5.6]{BaudoinGordinaMelcher2013}, $\frac{1}{2}L$ is the generator for $\{g_t\}_{t\ge0}$.

Now we define the heat kernel measure on $G$ as the end point distribution of a Brownian motion.

\begin{definition}
\label{df.HeatKernelMeasureInfinite}
We call a family of measures $\{\nu_t\}_{t>0}$ on $G$ defined by $\nu_t=\mathrm{Law}(g_{t})$ for any $t>0$ the {\em heat kernel measure
at time $t$}.
\end{definition}

Analogously to the heat kernel measure on non-isotropic Heisenberg groups, \cite[Corollary 5.7]{BaudoinGordinaMelcher2013} shows that $\{\nu_t\}_{t>0}$ satisfies the heat equation as follows
\begin{align*}
& \frac{d}{dt} \int_{G}f\left( g \right)d\nu_t\left(g \right)=\int_{G}\left(\frac{1}{2}Lf\right)\left( g \right)d\nu_t\left( g \right),
\\
&
\lim_{t \to 0}\int_{G}f\left( g \right)d\nu_t\left( g \right)=f(e)
\end{align*}
for any $t>0$ and any cylinder polynomial $f$.

We include the following proposition (see \cite[Proposition
2.30]{GordinaMelcher2013}) which states that, as the name suggests, the Cameron-Martin subgroup is a subspace of heat kernel measure 0.

\begin{proposition}[Proposition 2.30 in \cite{GordinaMelcher2013}]
For all $t>0$, $\nu_t(G_{CM})=0$.
\end{proposition}

Finally, we collect some results that connect $g_t$ and $\nu_t$ with the Brownian motion and the heat kernel measure on finite-dimensional projection groups. They can be found in \cite{BaudoinGordinaMelcher2013} and \cite{GordinaMelcher2013}.

\begin{notation}
For $P\in\mathrm{Proj}(W)$, let $g_t^P$ be the continuous process on $G_P$
defined by
\[ g_t^P = \left(PB_t, \frac{1}{2}\int_0^t\omega(PB_s,dPB_s)\right)\]
and let $\nu_t^P :=\mathrm{Law}(g_{t}^P)$.
\end{notation}

As stated in \cite[Proposition 5.4]{BaudoinGordinaMelcher2013}, $g_t^P$ is a Brownian motion on $G_P$ and $\{g_t\}_{t>0}$ can be approximated by the hypoelliptic Brownian motion on the finite-dimensional projection groups. In particular, for all $p\in[1,\infty)$ and $t>0$, for a family of increasing
projections $\{P_n\}_{n=1}^\infty\subset\mathrm{Proj}(W)$, we have
\begin{align} \label{eqn.bmapprox}
\lim_{n\rightarrow\infty} \mathbb{E}\left[\sup_{\tau\le t}
    \|g_\tau^{P_n}-g_\tau\|_\mathfrak{g}^p\right]=0.
\end{align}

For all projections satisfying H\"ormander's condition, the Brownian
motions on $G_P$ are subelliptic diffusions and thus
their laws are absolutely continuous with respect to the
finite-dimensional reference measure and their transition kernels
are smooth. By \cite[Lemma 2.27]{GordinaMelcher2013}, for all $P\in\operatorname*{Proj}(W)$ and $t>0$, we have
\begin{align} \label{eqn.HeatKernelMeasureProjection}
\nu_{t}^{P}(dx) = p_t^P(x)dx
\end{align}
where $dx$ is the Riemannian volume measure (equal to Haar measure) and
$p_{t}^{P}(x)$ is the heat kernel on $G_P$. This is consistent with Definition~\ref{df.HeatKernelMeasureNonisotropic}.

\subsection{Closability of the Dirichlet form and Logarithmic Sobolev inequalities on infinite-dimensional Heisenberg groups}

We  start by proving the closability of the Dirichlet form corresponding to the horizontal gradient.

\begin{theorem}
Given cylinder polynomials $u$, $v$ on $G$, let
\begin{align}
\mathcal{E}^{0}_t\left(u,v\right):=\int_G \langle \operatorname{grad}_Hu,\operatorname{grad}_Hv \rangle_H d\nu_t.
\end{align}
Then $\mathcal{E}^{0}_t$ is closable and its closure, $\mathcal{E}_t$, is a Dirichlet form on $L^2\left(G,\nu_t\right)$.
\end{theorem}

\begin{proof}
The closability of $\mathcal{E}^{0}_t$ is equivalent to the closability of the horizontal gradient operator $\operatorname{grad}_H: L^2\left(\nu_t\right) \rightarrow L^2\left(\nu_t\right) \otimes H$ with the domain $\mathcal{D}\left(\operatorname{grad}_H\right)$ being the space of cylinder polynomials on $G$. To check the latter statement, by \cite[Theorem VIII.1]{ReedSimonI}, it suffices to show that $\operatorname{grad}_{H}$ has a densely defined adjoint. For this we use an integration by parts formula for the hypoelliptic heat kernel measure $\nu_{t}$. Namely, for any $h\in H$ and any cylindrical polynomials $u$ and $v$ we have
\begin{align*}
\langle \operatorname{grad}_Hu,v\cdot h\rangle_{L^2\left(\nu_t\right) \otimes H} & =\int_{G}\widetilde{\left(h,0\right)}u \cdot v d\nu_t
\\
&
=\int_{G} \left(\widetilde{\left(h,0\right)}\left(u\cdot v\right)-u\cdot \widetilde{\left(h,0\right)}v\right)d\nu_t
\\
&
=\langle u, v\widetilde{\left(h,0\right)}^{*}1-\widetilde{\left(h,0\right)}v\rangle_{L^2\left(\nu_t\right)},
\end{align*}
where we used \eqref{eqn.HorizaontalGradientInfinite} in the first equality, the product rule in the second equality and the integration by parts formula for $\nu_t$ (see \cite[Corollary 8.10]{DriverEldredgeMelcher2016} for the explicit expression for $\widetilde{\left(h, 0\right)}^{*}$) in the third equality. This shows that $v\cdot h$ is in the domain of the adjoint of $\operatorname{grad}_H$. This completes the proof since functions of the form $v\cdot h$ are total in $L^2\left(\nu_t\right) \otimes H$. Therefore, we can see that the closure of $\mathcal{E}^0_t$, $\mathcal{E}_t$ is a Dirichlet form on $L^2\left(G,\nu_t\right)$.
\end{proof}

By \cite[Chapter IV, Section 4b]{MaRocknerBook} or \cite[Proposition 3.1]{RocknerSchmuland1992}, such a closed Dirichlet form $\mathcal{E}_t$ that we constructed is quasi-regular, which allows us to study the associated process in the infinite-dimensional setting we are considering. In this section,  we will extend our consideration of functions to a wider class  and then prove a hypoelliptic logarithmic Sobolev inequality on $G$.

\begin{definition}
We say that $G$ with the heat kernel measure $\nu_t$ satisfies a \emph{logarithmic Sobolev inequality with constant $C\left(\omega, t\right)$} if
\begin{align} \label{LSIInfinite}
\int_{G}f^2\log f^2d\nu_t-\left(\int_{G}f^2 d\nu_t\right)\log\left(\int_{G}f^2d\nu_t\right)
\leqslant C\left(\omega, t\right) \mathcal{E}_t\left(f,f\right)
\end{align}
for any $f\in \mathcal{D}\left(\mathcal{E}_t\right)$ and any $t>0$. In such a case we also say that $\operatorname{LSI}\left(C\left(\omega, t \right), \nu_t\right)$ holds.
\end{definition}

We can now prove a hypoelliptic logarithmic Sobolev inequality on $G$.

\begin{theorem} \label{thm.LSIInfinite}
The logarithmic Sobolev inequality $\operatorname{LSI}\left( C\left(\omega,t\right), \nu_t \right)$ holds on $G$ where the logarithmic Sobolev constant can be chosen to be $C\left(\omega,t\right)=C\left(\omega_0\right)t$.
\end{theorem}
Before we proceed to the proof we observe that the logarithmic Sobolev constant is the same as in Theorem~\ref{thm.LSINonisotropic} independent of $\omega$.

\begin{proof}
Our proof uses an approximation argument which is similar to the elliptic case in \cite[Section 8.2]{DriverGordina2008}, even though we do not have uniform curvature bounds.

Any $f\in \mathcal{D}\left(\mathcal{E}_t\right)$ can be approximated by the case when the functions are cylinder polynomials by a standard  limiting argument similarly to  \cite[Example 2.7]{Gross1993b}, so it suffices to consider $f:G \rightarrow \mathbb{R}$ to be a cylinder polynomial as in Definition~\ref{d.cyl}. Let $\{P_n\}_{n=1}^\infty\subset\mathrm{Proj}(W)$ such that $P_n|_H\uparrow I_H$, then $\{G_{P_n}\}_{n\geqslant 1}$ is a family of non-isotropic Heisenberg groups. By Theorem~\ref{thm.LSINonisotropic} together with \eqref{eqn.HeatKernelMeasureProjection}, we have
\begin{align*}
& \int_{G_{P_n}}f^2\log f^2d\nu^{P_n}_t-\left(\int_{G_{P_n}}f^2d\nu^{P_n}_t\right)\log\left(\int_{G_{P_n}}f^2d\nu^{P_n}_t\right)
\\
&
\leqslant C\left(\omega_0\right)t\int_{G_{P_n}} \left\Vert \operatorname{grad}^{P_n}_Hf\right\Vert^2_H d\nu^{P_n}_t.
\end{align*}
Letting $n\to \infty$ in the above inequality, by \eqref{eqn.bmapprox} and the Dominated Convergence theorem, we can prove \eqref{LSIInfinite}. Thus we can take $C\left(\omega,t\right)=C\left(\omega_0\right)t$ which is the same as the constant for the non-isotropic finite-dimensional Heisenberg group, which is essentially the constant on the three-dimensional isotropic Heisenberg group, and therefore it is independent of $\omega$.
\end{proof}

\appendix

\section{Symplectic forms}\label{s.SymplSpace}

The exposition below is based on \cite[Chapter 2]{McDuffSalamonBook1998}. Suppose $V$ is a real vector space. In what follows we assume that all objects are defined over $\mathbb{R}$.

\begin{definition} A \emph{skew-symmetric bilinear form} is a bilinear form $\omega: V \times V \longrightarrow \mathbb{R}$  such that $\omega\left( v, w \right)=-\omega\left( w, v \right)$ for all $v, w \in V$.  A \emph{symplectic form} is a non-degenerate skew-symmetric bilinear form, that is, such a skew-symmetric bilinear form that if  $\omega\left( v, w \right)=0$ for all $v \in V$, then $w=0$. We call $V$ equipped with a symplectic form a \emph{symplectic space} $\left( V, \omega \right)$.
\end{definition}
Note that for a skew-symmetric form we have $\omega\left( v, v \right)=0$.

\begin{theorem}\label{t.SymplBasis} Suppose $\left( V, \omega \right)$ is a symplectic space. Then $\operatorname{dim} V$ is even and there exists a \emph{symplectic basis} of $V$, that is,

\begin{align*}
& \omega\left( p_{i}, q_{i} \right)=-\omega\left( q_{i}, p_{i}\right)=1,
\\
& \omega\left( p_{i}, q_{j} \right)=0, i\not=j,
\\
& \omega\left( p_{i}, p_{j} \right)=\omega\left( q_{i}, q_{j} \right)=0
\end{align*}
for $i, j =1, .., n$, where $\operatorname{dim} V=2n$.
\end{theorem}
Such a basis also is called an \emph{$\omega$-standard basis}. Observe that this notion does not require $V$ to be equipped with any inner product.

\begin{proof}
We prove it by induction on $\operatorname{dim} V$. The base case is evident due to the non-degeneracy of $\omega$. Assume now  $\operatorname{dim} V=n$ and assume the result holds for all vector spaces of dimension $n-2$. Suppose $q \in V$ is non-zero. The form $\omega$ is non-degenerate, there exists $p \in V$ such that $\omega\left( p, q \right)\not=0$. We can normalize $p$ and $q$ in such a way $\omega\left( p, q \right)=1$. Denote
\[
W:=\left\{ v \in V: \omega\left( v, p \right)=0\text{ and } \omega\left( v, q \right)=0 \right\}.
\]
We claim that $W\cap \operatorname{Span}\left\{ p, q \right\} = \left\{ 0 \right\}$. Indeed, suppose $v \in W \cap \operatorname{Span}\left\{ p, q \right\}$. Then $v=ap+bq$ for some $a, b \in \mathbb{R}$. Since $v \in W$, we have that $\omega\left( v, p \right)=0$. At the same time $\omega\left( v, p \right)=-b$, so $b=0$. Similarly $a=0$, hence $v=0$. Now we can use that $\omega$ is non-degenerate, so $\operatorname{dim} W + \operatorname{dim}\operatorname{Span}\left\{ p, q \right\} = \operatorname{dim} V$, and therefore $V=W \oplus \operatorname{Span}\left\{ p, q \right\}$.

To use the inductive hypotheses, we need to check that the restriction of $\omega$ to $W$ is a symplectic form. It is obviously skew-symmetric, so we just need to check that it is non-degenerate.  Take $w \in W, w\not=0$, then there is a $v \in V$ such that $\omega\left( v, w \right)\not=0$. Then we can write $w=v_{1}+v_{2}$ with $v_{1} \in W$ and $v_{2} \in \operatorname{Span}\left\{ p, q \right\}$. As $\omega\left( v_{2}, w \right)=0$, then $\omega\left( v_{1}, w \right)\not=0$, that is, $\omega$ is non-degenerate on $W$. We complete the proof by applying the inductive hypothesis to $W$ equipped with the restriction of the symplectic form $\omega$.
\end{proof}

Suppose that in addition to $\omega$ the vector space $V$ is equipped with an inner product. A different proof gives simultaneous normalization of the symplectic form $\omega$ and an inner product on $V$, and as a by-product the existence of a symplectic basis.

\begin{proposition}[Lemma 2.42 in \cite{McDuffSalamonBook1998}]\label{p.SymplBasis}
Suppose $\left( V, \omega \right)$ is a symplectic space of dimension $2n$, and $g: V \times V \longrightarrow \mathbb{R}$ is an inner product on $V$. Then there is a symplectic basis $\left\{ p_{i}, q_{j} \right\}_{i, j=1}^{n}$ such that it is $g$-orthogonal and

\[
g\left( p_{i}, p_{i} \right)=g\left( q_{i}, q_{i} \right), i=1, ..., n.
\]
\end{proposition}

\begin{proof} It is enough to show that for $\mathbb{R}^{2n}$  with the standard inner product there is an orthogonal basis diagonalizing a non-degenerate skew-symmetric form. We define a $2n \times 2n$ matrix $A$ by

\[
\omega\left( u, v \right)=:\langle u, Av \rangle.
\]
Then $A$ is non-degenerate and $A^{T}=-A$, and therefore $iA \in \mathbb{C}^{2n \times  2n}$ is Hermitian. This means that the
spectrum of $A$ is purely imaginary and there is an orthonormal basis of eigenvectors in  $\mathbb{C}^{2n}$ for $A$. That is, there are $\alpha_{j}>0$ and (orthonormal) $u_{j}+iv_{j} \in \mathbb{C}^{2n}$ such that $A\left( u_{j}+iv_{j} \right)=i \alpha_{j}\left( u_{j}+iv_{j} \right)$, and then $A\left( u_{j}-iv_{j} \right)=-i \alpha_{j}\left( u_{j}-iv_{j} \right)$. So $u_{j}-iv_{j}$ is an eigenvector for $-i \alpha_{j}$, and therefore it is orthogonal to the eigenvector $u_{j}+iv_{j}$ since $A$ is skew-symmetric. Thus we have

\begin{align*}
& A\left( u_{j}+iv_{j} \right)=i \alpha_{j}\left( u_{j}+iv_{j} \right), j=1, ..., n,
\\
& \left( u_{j}+iv_{j} \right)^{T}\left( u_{k}+iv_{k} \right)=\delta_{jk},
\\
& \left( u_{j}-iv_{j} \right)^{T}\left( u_{k}+iv_{k} \right)=0.
\end{align*}
By equating real and imaginary parts, we have

\begin{align*}
& A u_{j}=- \alpha_{j}v_{j}, j=1, ..., n,
\\
& A v_{j} = \alpha_{j} u_{j}, j=1, ..., n,
\\
& \left( u_{j}+iv_{j} \right)^{T}\left( u_{k}+iv_{k} \right)=\delta_{jk}.
\end{align*}
Then

\begin{align*}
& A u_{j}=- \alpha_{j}v_{j}, j=1, ..., n,
\\
& A v_{j} = \alpha_{j} u_{j}, j=1, ..., n,
\\
&  u_{j}^{T} u_{k}=u_{j}^{T} v_{k}=v_{j}^{T} v_{k}=0.
\end{align*}
This gives $\omega\left( u_{j}, v_{j} \right)=u_{j}^{T}A v_{j}=\alpha_{j}\vert u_{j}\vert^{2}>0$ and the rest of the identities needed to complete the proof.
\end{proof}
We call a symplectic space  $\left( V, \omega \right)$ \emph{isotropic} if such a symplectic basis is not only orthogonal with respect to the metric $g$, but orthonormal. Otherwise the space in non-isotropic and the lengths of the orthogonal basis are used to parameterize the symplectic form $\omega$ in Equation \eqref{e.SymplForm}.

\begin{acknowledgement}
The  authors would like to thank N.~Eldredge for helpful discussions during the preparation of this work.
\end{acknowledgement}

\providecommand{\bysame}{\leavevmode\hbox to3em{\hrulefill}\thinspace}
\providecommand{\MR}{\relax\ifhmode\unskip\space\fi MR }
\providecommand{\MRhref}[2]{%
  \href{http://www.ams.org/mathscinet-getitem?mr=#1}{#2}
}
\providecommand{\href}[2]{#2}


\begin{thebibliography}{10}

\bibitem{BakryBaudoinBonnefontChafai2008}
Dominique Bakry, Fabrice Baudoin, Michel Bonnefont, and Djalil Chafa{\"{\i}},
  \emph{On gradient bounds for the heat kernel on the {H}eisenberg group}, J.
  Funct. Anal. \textbf{255} (2008), no.~8, 1905--1938. \MR{2462581
  (2010m:35534)}

\bibitem{BakryGentilLedouxBook}
Dominique Bakry, Ivan Gentil, and Michel Ledoux, \emph{Analysis and geometry of
  {M}arkov diffusion operators}, Grundlehren der Mathematischen Wissenschaften
  [Fundamental Principles of Mathematical Sciences], vol. 348, Springer, Cham,
  2014. \MR{3155209}

\bibitem{BaudoinBonnefont2012}
Fabrice Baudoin and Michel Bonnefont, \emph{Log-{S}obolev inequalities for
  subelliptic operators satisfying a generalized curvature dimension
  inequality}, J. Funct. Anal. \textbf{262} (2012), no.~6, 2646--2676.
  \MR{2885961}

\bibitem{BaudoinFeng2015}
Fabrice Baudoin and Qi~Feng, \emph{Log-{S}obolev inequalities on the horizontal
  path space of a totally geodesic foliation}, 2015, arxiv preprint.

\bibitem{BaudoinGarofalo2017}
Fabrice Baudoin and Nicola Garofalo, \emph{Curvature-dimension inequalities and
  {R}icci lower bounds for sub-{R}iemannian manifolds with transverse
  symmetries}, J. Eur. Math. Soc. (JEMS) \textbf{19} (2017), no.~1, 151--219.
  \MR{3584561}

\bibitem{BaudoinGordinaMelcher2013}
Fabrice Baudoin, Maria Gordina, and Tai Melcher, \emph{Quasi-invariance for
  heat kernel measures on sub-{R}iemannian infinite-dimensional {H}eisenberg
  groups}, Trans. Amer. Math. Soc. \textbf{365} (2013), no.~8, 4313--4350.
  \MR{3055697}

\bibitem{BealsGaveauGreiner2000}
Richard Beals, Bernard Gaveau, and Peter~C. Greiner, \emph{Hamilton-{J}acobi
  theory and the heat kernel on {H}eisenberg groups}, J. Math. Pures Appl. (9)
  \textbf{79} (2000), no.~7, 633--689. \MR{1776501 (2001g:35047)}

\bibitem{BogachevGaussianMeasures}
Vladimir~I. Bogachev, \emph{Gaussian measures}, Mathematical Surveys and
  Monographs, vol.~62, American Mathematical Society, Providence, RI, 1998.
  \MR{MR1642391 (2000a:60004)}

\bibitem{BonfiglioliLanconelliUguzzoniBook}
A.~Bonfiglioli, E.~Lanconelli, and F.~Uguzzoni, \emph{Stratified {L}ie groups
  and potential theory for their sub-{L}aplacians}, Springer Monographs in
  Mathematics, Springer, Berlin, 2007. \MR{2363343}

\bibitem{BonnefontChafaiHerry2020}
Michel Bonnefont, Djalil Chafa\"{\i}, and Ronan Herry, \emph{On logarithmic
  {S}obolev inequalities for the heat kernel on the {H}eisenberg group}, Ann.
  Fac. Sci. Toulouse Math. (6) \textbf{29} (2020), no.~2, 335--355.
  \MR{4150544}

\bibitem{CorwinGreenleafBook}
Lawrence~J. Corwin and Frederick~P. Greenleaf, \emph{Representations of
  nilpotent {L}ie groups and their applications. {P}art {I}}, Cambridge Studies
  in Advanced Mathematics, vol.~18, Cambridge University Press, Cambridge,
  1990, Basic theory and examples. \MR{1070979 (92b:22007)}

\bibitem{BouDagherZegarlinski2021}
Esther~Bou Dagher and Boguslaw Zegarlinski, \emph{Coercive inequalities in
  higher-dimensional anisotropic heisenberg group}, 2021.

\bibitem{DriverEldredgeMelcher2016}
Bruce~K. Driver, Nathaniel Eldredge, and Tai Melcher, \emph{Hypoelliptic heat
  kernels on infinite-dimensional {H}eisenberg groups}, Trans. Amer. Math. Soc.
  \textbf{368} (2016), no.~2, 989--1022. \MR{3430356}

\bibitem{DriverGordina2008}
Bruce~K. Driver and Maria Gordina, \emph{Heat kernel analysis on
  infinite-dimensional {H}eisenberg groups}, J. Funct. Anal. \textbf{255}
  (2008), no.~9, 2395--2461. \MR{MR2473262}

\bibitem{DriverGrossSaloff-Coste2009a}
Bruce~K. Driver, Leonard Gross, and Laurent Saloff-Coste, \emph{Holomorphic
  functions and subelliptic heat kernels over {L}ie groups}, J. Eur. Math. Soc.
  (JEMS) \textbf{11} (2009), no.~5, 941--978. \MR{2538496 (2010h:32052)}

\bibitem{DriverGrossSaloff-Coste2010}
\bysame, \emph{Growth of {T}aylor coefficients over complex homogeneous
  spaces}, Tohoku Math. J. (2) \textbf{62} (2010), no.~3, 427--474.
  \MR{2742018}

\bibitem{DriverMelcher2005}
Bruce~K. Driver and Tai Melcher, \emph{Hypoelliptic heat kernel inequalities on
  the {H}eisenberg group}, J. Funct. Anal. \textbf{221} (2005), 340--365.

\bibitem{Eldredge2010}
Nathaniel Eldredge, \emph{Gradient estimates for the subelliptic heat kernel on
  {$H$}-type groups}, J. Funct. Anal. \textbf{258} (2010), no.~2, 504--533.
  \MR{2557945 (2011d:35217)}

\bibitem{FrankLieb2012}
Rupert~L. Frank and Elliott~H. Lieb, \emph{Sharp constants in several
  inequalities on the {H}eisenberg group}, Ann. of Math. (2) \textbf{176}
  (2012), no.~1, 349--381. \MR{2925386}

\bibitem{Fraser2001a}
A.~J. Fraser, \emph{An {$(n+1)$}-fold {M}arcinkiewicz multiplier theorem on the
  {H}eisenberg group}, Bull. Austral. Math. Soc. \textbf{63} (2001), no.~1,
  35--58. \MR{1812307}

\bibitem{GordinaMelcher2013}
M.~Gordina and T.~Melcher, \emph{A subelliptic {T}aylor isomorphism on
  infinite-dimensional {H}eisenberg groups}, Probability Theory and Related
  Fields \textbf{155} (2013), 379--426.

\bibitem{Gordina2017}
Maria Gordina, \emph{An application of a functional inequality to
  quasi-invariance in infinite dimensions}, pp.~251--266, Springer New York,
  New York, NY, 2017.

\bibitem{GordinaLaetsch2016a}
Maria Gordina and Thomas Laetsch, \emph{Sub-{L}aplacians on {S}ub-{R}iemannian
  {M}anifolds}, Potential Anal. \textbf{44} (2016), no.~4, 811--837.
  \MR{3490551}

\bibitem{Gross1975c}
Leonard Gross, \emph{Logarithmic {S}obolev inequalities}, Amer. J. Math.
  \textbf{97} (1975), no.~4, 1061--1083. \MR{MR0420249 (54 \#8263)}

\bibitem{Gross1992}
\bysame, \emph{Logarithmic {S}obolev inequalities on {L}ie groups}, Illinois J.
  Math. \textbf{36} (1992), no.~3, 447--490. \MR{1161977 (93i:22012)}

\bibitem{Gross1993b}
\bysame, \emph{Logarithmic {S}obolev inequalities and contractivity properties
  of semigroups}, Dirichlet forms (Varenna, 1992), Lecture Notes in Math., vol.
  1563, Springer, Berlin, 1993, pp.~54--88. \MR{MR1292277 (95h:47061)}

\bibitem{GuionnetZegarlinski2003}
A.~Guionnet and B.~Zegarlinski, \emph{Lectures on logarithmic {S}obolev
  inequalities}, S\'eminaire de {P}robabilit\'es, {XXXVI}, Lecture Notes in
  Math., vol. 1801, Springer, Berlin, 2003, pp.~1--134. \MR{1971582}

\bibitem{HallLieBook}
Brian~C. Hall, \emph{Lie groups, {L}ie algebras, and representations}, Graduate
  Texts in Mathematics, vol. 222, Springer-Verlag, New York, 2003, An
  elementary introduction. \MR{1997306 (2004i:22001)}

\bibitem{HebischZegarlinski2010}
W.~Hebisch and B.~Zegarli\'nski, \emph{Coercive inequalities on metric measure
  spaces}, J. Funct. Anal. \textbf{258} (2010), no.~3, 814--851. \MR{2558178}

\bibitem{Hormander1967a}
Lars H{\"o}rmander, \emph{Hypoelliptic second order differential equations},
  Acta Math. \textbf{119} (1967), 147--171. \MR{0222474 (36 \#5526)}

\bibitem{HuLi2010}
Jun-Qi Hu and Hong-Quan Li, \emph{Gradient estimates for the heat semigroup on
  {H}-type groups}, Potential Anal. \textbf{33} (2010), no.~4, 355--386.
  \MR{2726903 (2012c:43010)}

\bibitem{KuoLNM1975}
Hui~Hsiung Kuo, \emph{Gaussian measures in {B}anach spaces}, Springer-Verlag,
  Berlin, 1975, Lecture Notes in Mathematics, Vol. 463. \MR{MR0461643 (57
  \#1628)}

\bibitem{LiHong-Quan2006}
Hong-Quan Li, \emph{Estimation optimale du gradient du semi-groupe de la
  chaleur sur le groupe de {H}eisenberg}, J. Funct. Anal. \textbf{236} (2006),
  no.~2, 369--394. \MR{MR2240167 (2007d:58045)}

\bibitem{LiHong-Quan2007}
\bysame, \emph{Estimations asymptotiques du noyau de la chaleur sur les groupes
  de {H}eisenberg}, C. R. Math. Acad. Sci. Paris \textbf{344} (2007), no.~8,
  497--502. \MR{MR2324485}

\bibitem{LiHong-QuanZhang2019}
Hong-Quan Li and Ye~Zhang, \emph{Revisiting the heat kernel on isotropic and
  nonisotropic {H}eisenberg groups}, Comm. Partial Differential Equations
  \textbf{44} (2019), no.~6, 467--503. \MR{3946611}

\bibitem{MaRocknerBook}
Zhi~Ming Ma and Michael R\"{o}ckner, \emph{Introduction to the theory of
  (nonsymmetric) {D}irichlet forms}, Universitext, Springer-Verlag, Berlin,
  1992. \MR{1214375}

\bibitem{McDuffSalamonBook1998}
Dusa McDuff and Dietmar Salamon, \emph{Introduction to symplectic topology},
  second ed., Oxford Mathematical Monographs, The Clarendon Press, Oxford
  University Press, New York, 1998. \MR{1698616}

\bibitem{ReedSimonI}
Michael Reed and Barry Simon, \emph{Methods of modern mathematical physics.
  {I}}, second ed., Academic Press Inc. [Harcourt Brace Jovanovich Publishers],
  New York, 1980, Functional analysis. \MR{MR751959 (85e:46002)}

\bibitem{RocknerSchmuland1992}
Michael R\"{o}ckner and Byron Schmuland, \emph{Tightness of general {$C_{1,p}$}
  capacities on {B}anach space}, J. Funct. Anal. \textbf{108} (1992), no.~1,
  1--12. \MR{1174156}

\bibitem{Schechtman2003a}
Gideon Schechtman, \emph{Concentration results and applications}, Handbook of
  the geometry of {B}anach spaces, {V}ol. 2, North-Holland, Amsterdam, 2003,
  pp.~1603--1634. \MR{1999604}

\end{thebibliography}
\end{document}